\documentclass[oneside]{amsart}
\usepackage{amsmath}
\usepackage[utf8]{inputenc}
\usepackage[T1]{fontenc}
\usepackage[dvipsnames]{xcolor}
\usepackage{todonotes}
\usepackage{graphicx} % Required for inserting images
\usepackage{hyperref}
\usepackage{mathrsfs,comment}
\usepackage{geometry}
\usepackage{float}
\usepackage{dsfont}
\usepackage{amssymb}
\usepackage{stmaryrd}
\usepackage[all,arc,2cell]{xy}
\usepackage{tikz-cd}
\usepackage{tikz-3dplot}
\usepackage[final]{showkeys}

\newcommand*{\Q}{{\mathbb Q}}

\newcommand*{\F}{{\mathbb F}}

\DeclareMathOperator{\colim}{colim}
\DeclareMathOperator{\Stab}{Stab}

\DeclareMathOperator{\Sub}{Sub}

\DeclareMathOperator{\gldim}{gl.dim}

\newcommand*{\m}[1]{{\protect\underline{#1}}}

\newcommand*{\mM}{\m{M}}

\newcommand{\cc}[1]{\mathcal #1}
\newcommand{\cC}{\cc{C}}

\newcommand{\cO}{\cc{O}}

\newcommand{\cP}{\cc{P}}

\newcommand{\Coeff}{\mathcal{C}\!\textit{oeff}}

\newcommand{\Ab}{\mathcal Ab}

\newcommand{\Mackey}{\mathcal Mackey}

\newcommand{\dashMod}{\mhyphen\mathcal Mod}

\newcommand{\Fin}{\mathcal{F}\!\textit{in}}
\newcommand*{\cA}{\mathcal A}
\newcommand{\OMackey}{\cO\mhyphen\Mackey}
\newcommand{\OpMackey}{\cO'\mhyphen\Mackey}

\mathchardef\mhyphen=45

\newcommand{\Span}{\textnormal{Span}}

\counterwithin{equation}{section}

\newtheorem{theorem}[equation]{Theorem}
\newtheorem{lemma}[equation]{Lemma}
\newtheorem{corollary}[equation]{Corollary}

\newtheorem{proposition}[equation]{Proposition}
\newtheorem{conjecture}[equation]{Conjecture}
\newtheorem*{theorem*}{Theorem}
\newtheorem*{proposition*}{Proposition}
\newtheorem*{lemma*}{Lemma}

\newtheorem{MainTheorems}{Theorem}

\theoremstyle{remark}
\newtheorem{remark}[equation]{Remark}
\newtheorem{example}[equation]{Example}

\theoremstyle{definition}

\newtheorem{definition}[equation]{Definition}

\DeclareMathOperator{\CoInd}{\textnormal{\textsf{CoInd}}}
\DeclareMathOperator{\Ind}{\textnormal{\textsf{Ind}}}

\newcommand*{\IndH}[1][{[H]}]{\Ind_{#1}}
\newcommand*{\CoIndH}[1][{[H]}]{\CoInd_{#1}}

\newcommand{\Id}{\mathrm{Id}}

\newcommand{\Mdef}[2]{\newcommand{#1}{{#2}}}

\Mdef{\bA}{\mathbb{A}}
\Mdef{\bB}{\mathbb{B}}
\Mdef{\bC}{\mathbb{C}}
\Mdef{\bD}{\mathbb{D}}
\Mdef{\bE}{\mathbb{E}}
\Mdef{\bF}{\mathbb{F}}
\Mdef{\bG}{\mathbb{G}}
\Mdef{\bH}{\mathbb{H}}
\Mdef{\bI}{\mathbb{I}}
\Mdef{\bJ}{\mathbb{J}}
\Mdef{\bK}{\mathbb{K}}
\Mdef{\bL}{\mathbb{L}}
\Mdef{\bM}{\mathbb{M}}
\Mdef{\bN}{\mathbb{N}}
\Mdef{\bO}{\mathbb{O}}
\Mdef{\bP}{\mathbb{P}}
\Mdef{\bQ}{\mathbb{Q}}
\Mdef{\bR}{\mathbb{R}}
\Mdef{\bS}{\mathbb{S}}
\Mdef{\bT}{\mathbb{T}}
\Mdef{\bU}{\mathbb{U}}
\Mdef{\bV}{\mathbb{V}}
\Mdef{\bW}{\mathbb{W}}
\Mdef{\bX}{\mathbb{X}}
\Mdef{\bY}{\mathbb{Y}}
\Mdef{\bZ}{\mathbb{Z}}
\Mdef{\cL}{\mathcal{L}}
\Mdef{\mcA}{\mathcal{A}}
\Mdef{\mcB}{\mathcal{B}}
\Mdef{\mcD}{\mathcal{D}} %
\Mdef{\mcE}{\mathcal{E}}
\Mdef{\mcF}{\mathcal{F}}
\Mdef{\mcG}{\mathcal{G}}
\Mdef{\mcH}{\mathcal{H}} 
\Mdef{\mcI}{\mathcal{I}}
\Mdef{\mcJ}{\mathcal{J}}
\Mdef{\mcK}{\mathcal{K}}
\Mdef{\mcL}{\mathcal{L}}
\Mdef{\mcM}{\mathcal{M}}
\Mdef{\mcN}{\mathcal{N}}
\Mdef{\mcO}{\mathcal{O}}
\Mdef{\mcP}{\mathcal{P}}
\Mdef{\mcQ}{\mathcal{Q}}
\Mdef{\mcR}{\mathcal{R}}
\Mdef{\mcS}{\mathcal{S}}
\Mdef{\mcT}{\mathcal{T}}
\Mdef{\mcU}{\mathcal{U}}
\Mdef{\mcV}{\mathcal{V}}
\Mdef{\mcW}{\mathcal{W}}
\Mdef{\mcX}{\mathcal{X}}
\Mdef{\mcY}{\mathcal{Y}}
\Mdef{\mcZ}{\mathcal{Z}}

\DeclareMathOperator{\ext}{Ext}
\DeclareMathOperator{\orb}{Orb}
\DeclareMathOperator{\ev}{Ev}
\DeclareMathOperator{\aut}{Aut}

\newcommand{\PreSh}{\mathcal{P}\!\textit{re}\mathcal{S}\!\textit{h}}

\newcommand{\SubOG}{\Sub_{G}^{\cO}}
\newcommand{\op}{\textrm{op}}

\newcommand{\abs}[1]{\left\vert#1\right\vert}

\title[Global dimension of \(\OMackey^G_\bQ\)]{Global dimension of the category of rational incomplete Mackey functors for a finite abelian group \(G\)}

\author{David Barnes}
\address{School of Mathematics and Physics, Queen's University Belfast, UK}
\email{d.barnes@qub.ac.uk}

\author{Anna Marie Bohmann}
\address{Department of Mathematics, Vanderbilt University, 1326 Stevenson Center, Nashville, TN, 37240, USA}
\email{am.bohmann@vanderbilt.edu}

\author{Michael A.~Hill}
\address{School of Mathematics, University of Minnesota, 206 Church Street, Minneapolis, MN, 55415 USA}
\email{mahill@umn.edu}

\author{Magdalena K\k{e}dziorek}
\address{Department of Mathematics, Radboud University Nijmegen, The Netherlands}
\email{m.kedziorek@math.ru.nl}

\begin{document}

\begin{abstract}In this paper, we analyse the global dimension of the category of rational incomplete Mackey functors over a finite abelian group.   Incomplete Mackey functors have recently risen to prominence in algebraic topology and hence it is valuable to understand their homological algebra invariants.  When working over the rational numbers, results of Greenlees--May and Th\'evanez--Webb show that the homological algebra of complete Mackey functors is quite simple, but the incomplete case is more complicated.  In this paper we use splitting results by the first, third and fourth authors to give an upper bound on the global dimension of rational incomplete Mackey functors where the incompleteness is governed by what is known as a disk-like transfer system. We then avail ourselves of a new connection to incidence algebras over posets to calculate the global dimension of rational incomplete Mackey functors in the disk-like case when the group is abelian.
\end{abstract}

\maketitle

\tableofcontents

\section{Introduction}
\subsection*{Background}
The global dimension of an abelian category is a classical invariant from homological algebra that gives a useful measure of complexity of the category. For example, the category of modules over a field has global dimension 0, the category of abelian groups has global dimension 1 and for \(k\) a field, the category of modules over \(k[x_1, \dots, x_n]\) has global dimension \(n\).  Low global dimensions usually indicate categories with simpler (homological) structure: for a ring \(R\), the category of left \(R\)-modules has global dimension at most one if and only if the ring \(R\) is left hereditary, meaning all submodules of projective left modules are projective. More generally, an abelian category is semi-simple if and only if the global dimension is 0. While there are several equivalent definitions, for our purposes the global dimension of an abelian category \(\mcA\) with enough injectives and enough projectives is the largest natural number \(n\) such that the bifunctor \(\ext_\mcA^n\) is non-zero, or infinity if no such number exists. 

In this paper, we are interested in abelian categories arising from the algebra of group actions. For example, when we consider the representation theory of a finite group \(G\), Maschke's theorem implies that the category of representations of \(G\) over a field of characteristic zero has global dimension 0.  Similarly, the category of \(G\)-Mackey functors over \(k\), a field of  characteristic zero, has global dimension zero. 
This result follows from the classification result of  Greenlees and May \cite{gremay95} and Th\'evanez and Webb \cite{tw90}, which splits the category of \(G\)-Mackey functors over \(k\) as the product of  categories of \(k[W_G H]\)-modules as \(H\) runs over  the conjugacy classes of subgroups of \(G\). 

We may think of these examples as the most basic case, generalizations of which usually give non-zero global dimensions. For example, in the case of profinite groups,  the global dimension of rational \(G\)-Mackey functors is related to the Cantor--Bendixson rank of the space of closed subgroups of \(G\); see work of the first author and Sugrue,  \cite{BSmackey}. In particular, when \(G\) is the \(p\)-adic integers, the global dimension is 1.  Generalising in a different direction, Bouc, Stancu and Webb \cite{BSW17} study the global dimension of the category of \(G\)-Mackey functors over a field of characteristic \(p\),  where the global dimension is sometimes infinite.
They give explicit conditions for when the category of cohomological Mackey functors has finite global dimension. Further results on homological algebra of Mackey functors can be found in Mart\'inez-P\'erez and Nucinkis \cite{MPN06}.

In this paper, the generalization is to incomplete Mackey functors: Mackey functors that do not have all transfer maps.  For \(G\) a finite group, \(G\)-incomplete Mackey functors  are the additive version of the bi-incomplete \(G\)-Tambara functors of Blumberg and the third author \cite{BH-Bi-incompleteT}. They naturally arise as homotopy groups of rational incomplete \(G\)-spectra. The transfer maps which are present are determined by  a \(G\)-transfer system \(\cO\); see Definition \ref{def:transfersystem}.  To capture that in the notation we will talk about \(\cO\)-Mackey functors. In the case of the complete \(G\)-transfer system, one recovers the classical case of \(G\)-Mackey functors.  In the case of the minimal \(G\)-transfer system, which has no non-identity maps, one recovers  the case of \(G\)-coefficient systems. 

This paper is part of a series of papers on rational \(\cO\)-Mackey functors.  In \cite{paper1}, the first, third and fourth author  proved splitting results for rational \(\cO\)-Mackey functors which generalize that of Greenlees and May \cite{gremay95} and Th\'evanez and Webb \cite{tw90}. In work in progress \cite{paper3}, the authors apply these results to  rational \(\cO\)-spectra, providing both an Adams spectral sequence to calculate maps of spectra in terms of rational \(\cO\)-Mackey functors, and an algebraic model: an equivalence of infinity-categories between  rational \(\cO\)-spectra and chain complexes of rational \(\cO\)-Mackey functors.  This gives an extrinsic motivation to calculate the global dimension of  the category of rational \(\cO\)-Mackey functors: if the global dimension is known to be \(k\), then one knows the Adams spectral sequence will collapse at page \((k+1)\),  bounding the complexity of the calculation.

\subsection*{Main results}

Our main result on global dimension of rational \(\cO\)-Mackey functors assumes that the group \(G\) is abelian and the transfer system \(\cO\) is disk-like:  that is, \(\cO\) is generated by some set of transfers  \(\{H \rightarrow G\}\), where \(H\leq G\). It is phrased in terms of the dimension of the inseparability classes of \(\cO\) (see Definitions \ref{def:insep} and \ref{def:n_h}). Recall that \(H \in {\SubOG}\) if \(H\rightarrow G\) in \(\cO\).

\begin{MainTheorems}[{Theorem \ref{thm:main_thm_disklike}}]
    Let \(\cO\) be a disk-like transfer system for a finite abelian group \(G\). The global dimension of the category of rational \(\cO\)-Mackey functors is the maximum over all subgroups \(H \in {\SubOG}\) of \(\dim ([H]^\cO)\) (see Definition \ref{def:n_h}). 
\end{MainTheorems}

To prove this theorem, we apply work of the first, third and fourth author  on splittings of rational \(\cO\)-Mackey functors in \cite{paper1}.  That paper shows that the category of rational \(\cO\)-Mackey functors for disk-like \(\cO\) splits into categories that look like coefficient systems---meaning there are no transfer maps, only restrictions and \(G\)-actions---on subcategories of \(\orb_G\). Hence, we consider the case of coefficient systems (the case of the trivial transfer system) for \(G\) in detail. 

\begin{MainTheorems}[{Corollary \ref{cor:gldim_coeffsyst}}]
    For \(G\) a finite abelian group, the global dimension of rational \(G\)-coefficient systems is the number of prime-power cyclic factors of \(G\).   
\end{MainTheorems}

This proof makes use of a surprising link to incidence algebras over posets.  The rational incidence algebra of a poset is the quotient of the rational path algebra where parallel paths are identified; see Spiegel and O'Donnell, \cite{SpiegelODonnell}. The link is that if one forgets the conjugation maps on a coefficient system,  then one obtains a rational presheaf on \(\Sub_G\), which is precisely the data  of a module over the rational incidence algebra of \(\Sub_G\).  In order to use this link to prove the theorem, we assume the group is abelian, as that simplifies the structure of the conjugation maps of the coefficient system. 
Thus, we are interested in the global dimension of rational incidence algebra of \(\Sub_G\),  for which we use the methods and results of Igusa and Zacharia, \cite{IZ90}.  Interestingly, there are few general formulas for the global dimension of incidence algebras over a poset. However, we obtain the following, which we believe to be new. 

\begin{MainTheorems}[{Theorem \ref{thm:gldim_abelian_group}}]
    For \(G\) a finite abelian group, the global dimension of the rational incidence algebra on \(\Sub_G\) is the number of prime-power cyclic factors of \(G\).
\end{MainTheorems}

A final result in the disk-like abelian case is that the global dimension is  monotonically decreasing with regards to inclusions of transfer systems.  We conjecture the analogue holds in the general setting in Conjecture \ref{conj:monotone}. 

\begin{MainTheorems}[{Theorem \ref{thm:monotonicity}}]
Suppose \(\cO_1 \le \cO_2\) is an inclusion of disk-like transfer systems for an abelian group \(G\). Then 
\[
\gldim(\cO_2\mhyphen\Mackey_\bQ^G) \leq \gldim (\cO_1\mhyphen\Mackey_\bQ^G).
\]
\end{MainTheorems}

In the case of arbitrary finite groups and disk-like transfer systems  we obtain an upper bound for the global dimension in terms of the maximum height of the inseparability classes; see Proposition~\ref{prop:uppedbound}.

\subsection*{Structure of the paper}
In Section \ref{sec:background}, we review the relevant definitions and the splitting result of \cite{paper1}. 
We start with the simplest case, the trivial transfer system, and provide the upper bound of
Proposition \ref{prop:uppedbound} for coefficient systems in Section \ref{sec:upperboundcoeff}.
In Section \ref{sec:incalg} we study incidence algebras and we relate them to coefficient systems using the 
adjunction results of Section \ref{sec:adjunctions}. Using our relation to incidence algebras 
we prove three of our main results in Section \ref{sec:gldimmackey}.
We apply these results to obtain the monotonicity result in Section \ref{sec:monotonicity} 
and further provide an alternative proof of monotonicity independent of the global dimension calculation.
We end the paper with three conjectures that arise from our approach in Section \ref{sec:conjectures}
and additional illustrative examples in Section \ref{sec:examples}.

\subsection*{Acknowledgements} 
The project was begun at a SQuaRE at the American Institute of Mathematics. The authors thank AIM for providing a supportive and mathematically rich environment.
The authors would like to thank the Isaac Newton Institute for Mathematical Sciences, Cambridge, for support and hospitality during the programme Equivariant Homotopy Theory in Context where work on this paper was undertaken. This work was supported by EPSRC grant no EP/Z000580/1 and the Simons Foundation, Award SFI-MPS-T-Institutes-00006117. The second author was partially supported by the United States National Science Foundation through grants DMS 2104300 and 2052849. The fourth author was supported by the Nederlandse Organisatie voor Wetenschappelijk Onderzoek (Dutch Research Council) Vidi grant no VI.Vidi.203.004.

\section{Definitions and a splitting result}\label{sec:background}
\begin{definition}\label{def:subg}
    Let \(G\) be a finite group. The poset of subgroups of \(G\) and inclusions will be denoted \(\Sub_G\).  
    The quotient of this poset under the action of \(G\) by conjugation will be denoted \({\Sub_G}/G\). 
    This poset has objects the conjugacy classes of subgroups of \(G\) and morphisms \((K) \to (H)\)
    whenever some conjugate of \(K\) is contained in \(H\).
\end{definition}    

\begin{definition}\label{def:transfersystem}
    Let \(G\) be a finite group. A \emph{\(G\)-transfer system} \(\cO\) is a partial order on \(\Sub_G\), denoted by arrows \(\to\), which refines the subset relation and which is closed under 
    \begin{itemize}
        \item conjugation: if \(K\to H\), then 
        \(gKg^{-1} \to gHg^{-1}\) for every
group element \(g \in G\), and
        \item restriction: if \(K\to H\) and \(L \subseteq H\), then \(K \cap L \to L\).
    \end{itemize}
 If  \(K\to H\), then we say that the \(H\)-set \(H/K\) is \(\cO\)-\emph{admissible}.
\end{definition}

Associated to any transfer system is a naturally defined version of Mackey functors where the transfers are parameterized by the maps in the transfer system. To make this precise, we recall a version of the Lindner category  starting with a subcategory \(\Fin^{G}_{\cO}\) of the category of finite \(G\)-sets \(\Fin^G\) that structures the transfers. In the language of \cite{BH-Bi-incompleteT}, 
\(\Fin^{G}_{\cO}\) is the indexing category
associated to \(\cO\).
\begin{definition}
    Let \(\Fin^{G}_{\cO}\) be the wide subcategory of the category \(\Fin^G\) of finite \(G\)-sets where a map \(f\colon S\to T\) of finite \(G\)-sets is in \(\Fin^G_{\cO}\) if and only if for all \(s\in S\), the following map is
    in \(\cO\):
    \[
        \Stab(s)\to \Stab\!\left(f(s)\right).
    \]
\end{definition}
This defines the smallest wide, pullback stable, finite coproduct complete subcategory of the category of finite \(G\)-sets that contains all morphisms \(G/K \to G/H\) whenever \(K \to H\) is in \(\cO\).

\begin{definition}[{\cite[Definition 7.21]{BH-Bi-incompleteT}}]
Define \(\cA^{\cO}\), the Lindner category of \(\cO\), to have objects the finite \(G\)-sets and morphisms the isomorphism classes of spans 
\[
\cA^{\cO}(S,T) = \big\{ [S \leftarrow U \xrightarrow{h} T] \mid h \in \Fin^{G}_{\cO}\big\}.
\]
\end{definition}

The category \(\cA^{\cO}\) is pre-additive, and the categorical biproduct is given by the disjoint union of finite \(G\)-sets.
When we take the complete transfer system consisting of all subgroup inclusions, 
\(\cA^{\cO}\) is the usual Lindner category of spans of \(G\)-sets.

\begin{definition}[{\cite[Definitions 7.21 and 7.24]{BH-Bi-incompleteT}}]\label{def:OMackeyFunctors}
    Let \(G\) be a finite group and \(\cO\) be a transfer system for \(G\). 
    An \(\cO\)-Mackey functor is a product preserving functor 
    \[
    M\colon \cA^{\cO} \to \Ab.
    \] 
    A map of \(\cO\)-Mackey functors is a natural transformation. We will use notation \(\OMackey^G\) for this category.
\end{definition}

The simplest context is where we consider Mackey functors with no  transfer maps, that is, incomplete Mackey functors for a trivial transfer system. These are known as coefficient systems.

Recall the \emph{orbit category} of \(G\), \(\orb_G\), is the category with objects the \(G\)-sets \(G/H\) for \(H\) a subgroup of \(G\) and morphisms the \(G\)-equivariant maps. Observe that when \(\cO\) is the trivial transfer system, the category \(\Fin_\cO^G\) is the finite coproduct completion of the maximal subgroupoid of \(\orb_G\). Thus product preserving functors out of \(\cA^{\cO}\) into \(\Ab\) are equivalent to contravariant functors from \(\orb_G\) to \(\Ab\), since \(\Ab\) is additive.

\begin{definition}\label{def:coeff}
    Let \(G\) be a finite group. A \emph{rational coefficient system} \(\mM\) for \(G\) is a  contravariant functor from the orbit category of \(G\) to \(\bQ\)-modules.  We denote the category of such objects and natural transformations by \(\Coeff_\bQ^{\,G}\).
\end{definition}

\begin{definition}\label{def:Sub_OG}
	If \(\cO\) is a transfer system, let \({\SubOG}\) be the sub-poset of \(\Sub_G\) consisting of subgroups \(H\) such that \(H\to G\) in \(\cO\). 
    As in Definition \ref{def:subg}, the quotient of this poset under the conjugation action of \(G\) will be denoted \(\SubOG/G\).
\end{definition}

When \(\cO\) is the complete transfer system, \(\SubOG = \Sub_G\) and  \({\SubOG}/G = {\Sub_G}/G\). 

\begin{definition}\label{def:insep}
    Two subgroups \(J\) and \(K\) are \emph{inseparable} if and only if for every \(L\in {\SubOG}\), we have
    \[
        \abs{(G/L)^J}=\abs{(G/L)^{K}}.
    \]
\end{definition}

Inseparability for \(\cO\) is an equivalence relation on \(\Sub_G\) and we  use the notation \([K]^{\cO}\) for the equivalence class of \(K\). Moreover, we know by \cite[Corollary 3.24]{paper1} that each inseparability class corresponds to a conjugacy class of a subgroup \(H\) such that \(H\to G\) in \(\cO\). Thus we can use a choice of \(H \in {\SubOG}\) to denote an inseparability class \([H]^\cO\).  Moreover, when \(G\) is abelian, this result also implies that when \(H\to G\), all elements of \([H]^\cO\) are subgroups of \(H\).
\begin{remark}\label{rmk:relfamily}
{\cite[Corollary 3.33]{paper1}} provides the following useful closure property of the inseparability classes $[H]^{\mcO}$: if $J\in[H]^{\mcO}$ and $J\subseteq K\subseteq H$, then $K\in[H]^{\mcO}$.  This follows from the fact that $[H]^{\mcO}$ is a relative family.
\end{remark}

\begin{example}\label{ex:C_pq_inseparability}
    Consider \(C_{pq}\) and the transfer system \(\cO\) consisting of transfers \(e \rightarrow G\), \(e\rightarrow C_{q}\), and \(e\rightarrow C_{p}\). Then there are two inseparability classes for \(\cO\): \([e]^\cO = \{e\}\) and \([C_{pq}]^\cO = \{C_{pq}, C_{q}, C_{p}\}\). Notice that all the subgroups except for \(e\) are inseparable from \(C_{pq}\).
\end{example}

\begin{definition}\label{defn:inseparableposets} 
Let \(H\to G \in \cO\). Define a sub-poset \(\Sub_{\langle H\rangle}^{\cO}\) of \(\Sub_G\) by
\[
    \Sub_{\langle H\rangle}^{\cO}=\{K\mid K\in [H]^{\cO}, K\subseteq H\}.
\]
 Let \(N=N_G(H)\) be the normalizer of \(H\) in \(G\). Denote by \(\orb^N_{\langle H\rangle^\cO}\) the full subcategory of the orbit category of \(N\) spanned by the orbits \(N/J\) with \(J\in \Sub^{\cO}_{\langle H\rangle}\); that is, 
    \[
        \orb^N_{\langle H\rangle^{\cO}}=\langle N/J\mid J\in [H]^{\cO}\text{ and }J\subseteq H\rangle.
    \]
    Notice that if \(G\) is an abelian group the conditions \(K\subseteq H\) and \(J\subseteq H\) above are vacuous. 
\end{definition}

\begin{example}
     Suppose \(G=C_{pq}\) and \(\cO\) is the transfer system from Example \ref{ex:C_pq_inseparability}. Then \(\Sub_{\langle e\rangle}^{\cO}=\{e\}\) and \(\Sub_{\langle G\rangle}^{\cO}\) is of the form 
\begin{center}
\begin{tikzpicture}[scale=0.5]
      \node (A) at (0,4) {\(C_{pq}\)};
	 \node (B) at (-2,2) {\(C_{q}\)};
	 \node (C) at (2,2) {\(C_{p}\).};
%      \node (D) at (0,0) {\(e}\)};
      \draw[<-,] (A) to (B) ;
	 \draw[<-,] (A) to (C) ;
%	 \draw[<-,thick] (B) to (D) ;
%	 \draw[<-,thick] (C) to (D) ;
\end{tikzpicture}
\end{center}
\end{example}

Extending Definition \ref{def:coeff} to this context we have the following. 
\begin{definition}
    Let \(\mcS\) be a full subcategory of \(\orb_G\). 
    Define \(\Coeff_{\mcS}^{\,G}\)
    to be the category with objects the contravariant functors from \(\mcS\) to  abelian groups and morphisms the natural transformations  
    \[
    \Coeff_{\mcS}^{\,G}=\{{\mcS}^{\text{op}}\to\Ab\}
    \]
    Write \(\Coeff_{\mcS, \bQ}^{\,G}\) for the category of contravariant functors with values in \(\bQ\)-modules.     
    We call such functors rational \(G\)-coefficient systems on \(\mcS\).
Taking the group to be \(N=N_GH\) and \(\orb^N_{\langle H\rangle^{\cO}}\) to be the subcategory of \(\orb_N\) of Definition \ref{defn:inseparableposets}, we set
\[
\Coeff_{\langle H\rangle^{\cO},\bQ}^{\,N}
=
\Coeff_{\orb^N_{\langle H\rangle^{\cO}},\bQ}^{\,N}.
\]
\end{definition}
These rational coefficient systems describe rational \(\cO\)-Mackey functors when the transfer system \(\cO\) is \emph{disk-like}.
\begin{definition}
    A transfer system \(\cO\) for a finite group \(G\) is \emph{disk-like} if it is generated by a set of transfers of the form \(H \to G\) for varying \(H\leq G\).
\end{definition}
The key point is that for these transfer systems, after the splitting there are no transfers in any inseparability class.

\begin{theorem}[{\cite[Theorem D.]{paper1}}]\label{thm:splittingIncompleteMackey}
     Let \(\cO\) be a disk-like transfer system for a group \(G\). There is an equivalence of categories
     \[
        \OMackey_\bQ^G \cong \prod_{(H)\in {\SubOG}/G} \Coeff_{\langle H \rangle^{\cO},\bQ}^{\,N_G(H)}.
    \]
\end{theorem}
As a choice of $(H)\in {\SubOG}/G$ determines an inseparability class, we refer to the rational coefficient systems $\Coeff_{\langle H\rangle^{\cO},\bQ}^{\,N}$ as \emph{rational coefficient systems on inseparability classes}.  When $G$ is abelian, these are all rational coefficient systems on subcategories of $\orb_G$.
\begin{corollary}\label{cor:abeliansplittingIncompleteMackey}
     Let \(\cO\) be a disk-like transfer system for an abelian group \(G\). There is an equivalence of categories
     \[
        \OMackey_\bQ^G \cong \prod_{H\in {\SubOG}} \Coeff_{\langle H \rangle^{\cO},\bQ}^{\,G}.
    \]
\end{corollary}
    
\section{An upper bound for the category of rational coefficient systems}\label{sec:upperboundcoeff}

In this section we give an upper bound on the global dimension of the category of rational \(\cO\)-Mackey functors 
when \(\cO\) is disk-like; see Proposition \ref{prop:uppedbound}. 
By Theorem \ref{thm:splittingIncompleteMackey} we can restrict to the case of coefficient systems 
on inseparability classes for \(\cO\). To obtain an upper bound on the global dimension of the category of coefficient systems
on an inseparability class, we connect projective resolutions to the poset structure of that inseparability class. 
We begin with the poset that we will use for the upper bound. 

Recall that a full subcategory is said to be \emph{replete} if whenever an object 
\(x\) is in the subcategory, so are all objects isomorphic to \(x\). 

\begin{definition} \label{def:projectiontoposet}
There is a projection \(\phi \colon \orb_G \to {\Sub_G}/G\) given by \(G/H \mapsto (H)\). 
The image of a full, replete subcategory 
\(\mcS \subseteq \orb_G\) under this projection is a sub-poset of \({\Sub_G}/G\), which we denote \(\widetilde{\mcS}\).
\end{definition} 
In this paper we always require that the subcategory \(\mcS\) is full and replete, 
the latter implying that if \(G/H\) is in \(\mcS\), then
so is \(G/K\) for every \(K \leq G\) which is conjugate to \(H\).

To find the upper bound we will need a set of projective objects. We obtain these via adjunctions.
The \emph{evaluation at \(H\) functor}
\[
\ev_H \colon \Coeff_{\mcS,\bQ}^{\,G} \to \bQ[W_G H]\dashMod
\]
has a left adjoint \(\underline{L}_{H}\)
and a right adjoint \(\underline{R}_{H}\) given by the following formulas. We use \( (-)^{W_G H}\) to indicate fixed points (also known as invariants):  
\begin{equation}\label{eqn:adjointstoevaluation}
\begin{split}
\underline{L}_H A(G/K)&=\Q \cdot \mcS^{\op}(G/H,G/K) \otimes_{\Q[W_GH]}A\\
\underline{R}_H A(G/K)&=\hom(\mcS^{\op}(G/K,G/H),A)^{W_GH}.
\end{split}
\end{equation}
Here we use the natural identification of $W_GH$ with $\mcS^{\op}(G/H,G/H)=\orb_G^{\op}(G/H,G/H)$ defined by sending $x\in N_GH $ to the $G$-map $gH\mapsto gxH$. These formulas for $\underline{L}_H$ and $\underline{R}_H$ follow from standard Kan extension constructions.

\begin{lemma}\label{lem:LHs Extension by Zero}
    If \(K\) is not subconjugate to \(H\), then 
    \[
        (\underline{L}_H A)(G/K)=0.
    \]
\end{lemma}
\begin{proof}
    If \(K\) is not subconjugate to \(H\), then
    \[
        \mcS^{op}(G/H,G/K)=\mcS(G/K,G/H)=\emptyset.  \qedhere
    \]
\end{proof}

As \(\ev_H\) is exact, its left adjoint \(\underline{L}_{H}\) preserves projectives and its right adjoint \(\underline{R}_{H}\)
preserves injectives. Moreover, 
\[
\ev_H \circ \underline{L}_{H} = \Id = \ev_H \circ \underline{R}_{H}.
\]

In the case of an abelian group \(G\), a coefficient system \(\mM\) can be thought of as a contravariant functor from the poset of subgroups of \(G\) to \(\bQ[G]\)-modules such that \(\mM(G/H)\), the value at \(H\), is \(H\)-fixed. 

\begin{remark}
Recall that the \emph{height} of a finite poset \(\mcP\) is the length of the longest 
chain of non-identity maps in \(\mcP\). That is, if the sequence of maps
\[
p_0 \to p_1 \to \cdots \to p_n
\]
is the longest possible chain in \(\mcP\), then \(\mcP\) has height \(n\).  
We also set here the convention that 
\(p \leqslant q\) is the statement that there is a map
\(p \to q\). Similarly, \(p < q\) is the statement that there is a map
\(p \to q\) and \(p \neq q\). 
\end{remark}

\begin{lemma}\label{lem:gdupperbound}
The global dimension of \(\Coeff_{\mcS,\bQ}^{\,G}\) is bounded above by the height of \(\widetilde{\mcS}\), the poset formed by projecting \(\mcS\) to \({\Sub_G}/G\).  
\end{lemma}
\begin{proof}
    If the poset  \(\widetilde{\mcS}\) has no maps, the global dimension is 0. 
    
    By Maschke's Theorem, all objects in \(\bQ[W_GH]\dashMod\) are projective.  As \(\underline{L}_H\) preserves projectives, examining the form of the left adjoints $\underline{L}_H$ shows that the map 
    \begin{equation}\label{eqn:bigprojection}
    \bigoplus_{(H)\in \widetilde{\mcS}} \underline{L}_{H} (\mM(G/H)) \to \mM
    \end{equation}
    is a surjection from a projective object.
    By Lemma \ref{lem:LHs Extension by Zero}, the value of $\underline{L}_H(\mM(G/H))$ at $G/K$ is zero when $K$ is not subconjugate to $H$.  Suppose $(K)$ is a maximal element of the support of $\mM$ in $\widetilde{\mcS}$. Then the  map \eqref{eqn:bigprojection} is an isomorphism at level $G/K$.   Hence the kernel of the map \eqref{eqn:bigprojection} has support concentrated in the support of $\mM$ with the maximal elements of the support of $\mM$ removed. The result then follows by induction over the height of the support of $\mM$. This is bounded above by the height of $\widetilde{\mcS}$. 
\end{proof}

We can now give the upper bound on the global dimension of categories of rational incomplete Mackey functors. 

\begin{proposition}\label{prop:uppedbound}
    The global dimension of the category of rational \(\cO\)-Mackey functors for a disk-like transfer system \(\cO\) for \(G\) is bounded above by the maximum height of the posets \(\Sub_{\langle H \rangle}^{\cO}/N_G(H)\), where the maximum runs over conjugacy classes of subgroups \(H\in {\SubOG}\).
\end{proposition}
\begin{proof}
    Using the splitting result in
    Theorem \ref{thm:splittingIncompleteMackey} for the category of rational \(\cO\)-Mackey functors for a disk-like transfer system \(\cO\), we get 
    \[
        \OMackey_\bQ^G \cong \prod_{(H)\in {\SubOG}/G} \Coeff_{\langle H \rangle^{\cO},\bQ}^{\,N_G(H)}.
    \]
    Thus the global dimension is the maximum of global dimensions of categories on the right hand side. Lemma \ref{lem:gdupperbound} gives an upper bound of this dimension in terms of the heights of the posets \(\Sub_{\langle H \rangle}^{\cO}/N_G(H)\),  as these are the projections of the 
    categories \(\orb^{N_G(H)}_{\langle H\rangle^{\cO}}\) to \(\Sub_{N_G(H)}/N_G(H)\) from Definition 
    \ref{defn:inseparableposets}.
\end{proof}

\begin{remark}\label{rmk:completecasebound}
    The complete transfer system for \(G\) is generated by the set of all maps \(H \to G\) 
    and hence it is disk-like. The splitting result in this case is that of 
    Greenlees and May \cite{gremay95} and Th\'evanez and Webb \cite{tw90},
    \[
        \Mackey_\bQ^G \cong \prod_{(H)\in {\Sub_G}/G} \bQ [W_GH] \dashMod.
    \]
    Here we see that the global dimension is zero, which agrees with the upper bound 
    because in this case, each poset is the singleton set.
\end{remark}

\begin{remark}
In general, the upper bound of Proposition \ref{prop:uppedbound} is rather poor. In the case of the trivial transfer system
and \(G=C_{p^n}\), the upper bound is \(n\) while the global dimension is 1, by Example \ref{ex:cyclicpower}. 
However, see Example \ref{ex:extcpq} for a group where this upper bound is the global dimension for the case of the trivial transfer system.
\end{remark}

\section{Adjunctions and global dimension}\label{sec:adjunctions}
Since we are discussing global dimension, throughout this section we assume that all categories are abelian categories and all functors considered are additive. 

\begin{lemma}\label{lem:adjunction_and_global_dim}
    Let \(F \colon \cC \to \mcD\) be a fully faithful exact functor between categories with enough projectives. 
    If \(F\) has an exact right adjoint, then for any $C,C'\in\cC$
    \[\ext_{\mcD}^*(FC,FC')=\ext_{\cC}^*(C,C').\]
    Thus the global dimension of \(\mcD\) is at least that of \(\cC\).
\end{lemma}
\begin{proof}
    Take a projective resolution \(P_\bullet\) of \(C \in \cC\). 
    As the right adjoint of \(F\) preserves surjections, \(F\) preserves projective objects. The exactness of \(F\) then implies that \(F\) preserves projective resolutions. Hence, \(FP_\bullet\) is a a projective resolution of \(FC \in \mcD\). 
    As \(F\) is fully faithful, for any \(C' \in \cC\) 
    \[
    \begin{array}{rcl}
    \ext_\mcD^*(FC, FC') 
    &=& H^* \hom_{\mcD} (FP_\bullet, FC') \\
    &=& H^* \hom_{\cC} (P_\bullet, C') \\
    &=&     \ext_\cC^*(C, C'). 
    \end{array}
    \]
    Hence, if \(\ext_\cC^n\) is non-zero, so is   \(\ext_\mcD^n\).  
\end{proof}

The analogous statement with injectives and an exact left adjoint holds by the same reasoning. The conditions of the lemma hold in the somewhat common case of the quintessential localisations of Johnstone \cite{Johnstone96}. 

\begin{definition}
A fully faithful functor \(F\colon \cC \to \mcD\) establishes \(\cC\) as a \emph{quintessential subcategory} of \(\mcD\) if \(F\) has an ambidextrous adjoint.
The adjoint functor to \(F\) is called a \emph{quintessential localisation} and we often  write this as saying \(F\) is \emph{part of a quintessential localisation}.
\end{definition}

\begin{corollary}\label{cor:quintlocal_gldim}
    If \(\cC\) is a quintessential subcategory of \(\mcD\) and \(\mcD\) has enough projectives, then \(\cC\) has enough projectives and 
    the global dimension of  \(\mcD\)
    is at least that of \(\cC\). 
\end{corollary}
\begin{proof}
    The functor \(i\colon \cC \to \mcD\) establishing \(\cC\) as a quintessential subcategory of \(\mcD\) is fully faithful and exact with an exact right adjoint \(r\). By assumption, for every \(c\in \cC\), there is a projective \(P\in \mcD\) and an epimorphism \(P \to i(c)\) in \(\mcD\). Since \(r\) preserves projective objects and epimorphisms, \(rP \to ri(c)\cong c\) shows that there are enough projectives in \(\cC\). The second statement follows from Lemma \ref{lem:adjunction_and_global_dim}.  
\end{proof}

\begin{example}
    For \(\cC\) and \(\mcD\) abelian categories, \(\cC\) is a quintessential subcategory of \(\cC\times \mcD\), using the functor \(i\colon \cC \longrightarrow \cC \times \mcD\) given by \(C\mapsto (C,0)\). This implies that \(\gldim \cC \leq \gldim (\cC\times\mcD)\).
\end{example}

\begin{example}
    For a finite group \(G\), the inflation functor from \(\bQ\)-modules to 
    \(\bQ[G]\)-modules giving a \(\bQ\)-module the trivial \(G\)-action
    is part of a quintessential localisation. The adjoints are taking
    orbits and fixed points, which are naturally isomorphic over the rationals.
\end{example}

We want to extend this example to coefficient systems. 
Recall from Definition \ref{def:subg} the poset \({\Sub_G}/G\) and 
recall from Definition \ref{def:projectiontoposet}
the projection functor \(\phi \colon \orb_G \to {\Sub_G}/G\)
which sends \(G/H\) to the conjugacy class \((H)\) of \(H\). 

\begin{definition}
    Let \(\mcS\) be a full, replete subcategory of \(\orb_G\), with corresponding poset
    \(\widetilde{\mcS} \subseteq  {\Sub_G}/G\).
    We refer to the category of contravariant functors from \(\widetilde{\mcS}\)  to \(\bQ \dashMod\) as
    \emph{presheaves over \(\widetilde{\mcS}\)} and denote it by \(\PreSh_\bQ (\widetilde{\mcS})\).

    A presheaf \(F\) over \(\widetilde{\mcS}\) defines an \(\mcS\)-coefficient system by precomposition with the projection functor \(\phi\). We denote this functor 
    \[\underline{\phi}^* \colon \PreSh_\bQ (\widetilde{\mcS}) 
    \longrightarrow 
    \Coeff_{\mcS,\bQ}^{\,G}.
    \]
\end{definition}

Given a presheaf \(F\), the value on objects is 
\[
(\underline{\phi}^* F) (G/H) = F((H))
\]
Given any map \(f \colon G/K \to G/H\) in \(\mcS\), the functor \(\phi\) sends this to the \emph{unique} map 
\(\phi(f) \colon (K) \to (H)\) in  \(\widetilde{\mcS}\). 
Hence
\[
(\underline{\phi}^* F )(f) = F(\phi(f)) \colon F((H)) \longrightarrow F((K)).
\]
In particular, the maps \(G/H \to G/H\) which give the conjugation maps at \(G/H\) are all sent to the identity so that \((\underline{\phi}^* F )(G/H)\) has the trivial action of \(N_G H\). 
It follows that any \(\underline{M}\) in the essential image of \(\underline{\phi}^*\)
has trivial \(N_G H\) action on \(\underline{M}(G/H)\) and has all maps 
\(\underline{M}(G/H) \to \underline{M}(G/K)\) the same.
    
\begin{lemma}
    The functor \(\underline{\phi}^*\) has both left and right adjoints. 
\end{lemma}    
\begin{proof}
    These adjoints are produced by left and right Kan extension. We may write these adjoints in their standard forms, as (co)ends or as (co)limits over slice categories
    \[
    \begin{array}{rcccl}
    \bL \mM ((H))  &=& \int^{G/K \in {\mcS}} \bQ \cdot \widetilde{\mcS}((H),(K))  \otimes \mM (G/K)  &=& \underset{(H) \to \phi(G/K) \in \mcS}{\colim} \mM(G/K) \\[9pt] 
    \bR \mM ((H))  &=& \int_{G/K \in {\mcS}} \hom\left( \widetilde{\mcS}((K),(H)),  \mM (G/K) \right) 
    &=& \underset{\phi(G/K) \to (H) \in \mcS}{\lim} \mM(G/K). \qedhere
    \end{array}
    \]    
\end{proof}

We would like this to be a quintessential localisation so that we can use earlier results of the section, 
but this is false in general. Indeed, the adjoints fail to be exact, 
so that we cannot use Lemma \ref{lem:adjunction_and_global_dim} either.

\begin{example}\label{ex:exactnessfails} 
We show the failure of exactness of \(\bL\). Let $G=A_4$, the alternating group on four letters, and let $K\subseteq A_4$ be the Klein four group. Let $C=\langle (12)(34)\rangle$ be one of the three subgroups of $K$ of order $2$ and let $\sigma C\sigma^{-1}$ and $\sigma^2 C\sigma^{-2}$ be its conjugates in $A_4$, where $\sigma=(123)$. Take \(\mcS\) to be the full subcategory on the orbits \(A_4/K\),  
    \(A_4/C\),
    \(A_4/\sigma^{-1} C\sigma\) and 
    \(A_4/\sigma^{-2} C\sigma^{2}\).

    The subgroup $K$ is normal in $A_4$ and its Weyl group is $W_G K\cong C_3$.  Let $V$ denote the $\Q[W_GK]$-module $\Q[W_GK]/\Q$, the reduced regular representation. We then obtain a coefficient system $M=\underline{L}_KV$ on $\mcS$ by applying the left adjoint $\underline{L}_K$ to evaluation at $K$ as defined in (\ref{eqn:adjointstoevaluation}).

We calculate the presheaf $\bL M$ using the adjunctions defining $\bL$ and $\underline{L}_K$.  For any presheaf $F$ on $\widetilde{\mcS}$, we have
\[
\PreSh_\bQ(\widetilde{\mcS})(\bL(\underline{L}_KV),F)\cong \Coeff_{\mcS,\Q}^{\,G}(\underline{L}_K V,\underline{\phi}^*F)\cong \hom_{\Q[W_GK]}\!\big(V,\ev_K\underline{\phi}^*F\big).
\]
By definition $\ev_K\underline{\phi}^*F=F((K))$ and, as discussed above, the action of $N_GK$ on $F((K))$ must be trivial.  Hence any $W_GK$-equivariant map from $V$ to $F((K))$ must factor through the coinvariants $(V)_{W_GK}$ of $V$, so we have 
\[\hom_{\Q[W_GK]}(V,F((K)))\cong \hom_\Q\!\big((V)_{W_GK},F((K))\big).\]
Since the coinvariants of \(V\) are zero, 
we find that $\PreSh_\Q(\widetilde{\mcS})(\bL M,F)\cong 0$ for all $F$. This implies that $\bL M$ is the constant presheaf at $0$.

Now consider the subobject $M'$ of $M$ that agrees with $M$ except at $G/K$, where $M'$ takes value $0$. The formula for $M=\underline{L}_KV$ from (\ref{eqn:adjointstoevaluation}) allows us to calculate $M'(G/C)=M(G/C)$ as
\[
M(G/C)=\Q[\mcS^\op(G/K,G/C)]\otimes_{\Q[W_GK]}V
\]
There are three maps in $\mcS$ from $G/C$ to $G/K$ which are permuted by the action of $W_GK$, and hence $M(G/C)=V$.  The normalizer of $C$ in $A_4$ is $K$ and so the Weyl group $W_GC$ is $K/C$.  Using the natural identification $\mcS^\op(G/C,G/C)\cong W_GC$, we see that this Weyl group acts on $M(G/C)$ via precomposition with the $\mcS^\op(G/K,G/C)$ term. Since $C$ is central in $K$ (indeed, $K$ is abelian),  the maps in $\mcS(G/C,G/C)$ commute with the canonical projection $G/C\to G/K$ and so the action of $W_GC$ on $M(G/C)=M'(G/C)$ is trivial. Similarly, $M'(G/\sigma^{-1} C\sigma)=V$ and $M'(G/\sigma^{-2} C\sigma^{2})=V$ with trivial Weyl actions.  

To calculate $\bL M'((C))$, we take the colimit of $M'$ over maps out of $(C)$.  Since $M'(G/K)=0$, we need only take the colimit over the full subcategory spanned by $G/C$, $G/\sigma^{-1}C\sigma$ and $G/\sigma^{-2}C\sigma^2$.  The endomorphisms of each of these act trivially on $M'$, so we see that 
\[
\bL M'((C))=V\neq 0.
\]
It follows that $\bL(M'\to M)$ is not a monomorphism and so $\bL$ is not exact.  A similar example may be constructed to show that $\bR$ is not exact.
\end{example}

\begin{proposition}\label{prop:quintlocalnoaction}
    For \(G\) an abelian group, the functor \(\underline{\phi}^*\) is a part of a quintessential localisation. 
\end{proposition}
\begin{proof}
    In the case of an abelian group the adjoints take the following form, 
    using \( (-)_{G/H} \) to indicate \(G/H\)-orbits (coinvariants) and
    \( (-)^{G/H} \) to indicate \(G/H\)-fixed points (invariants)
    \[
    \begin{array}{rcccl}
    \bL \mM (H)  &=& \mM (G/H)_{G/H} \\[9pt] 
    \bR \mM (H)  &=& \mM (G/H)^{G/H}.
    \end{array}
    \]
    We may deduce this form for the left adjoint at \(H\) from the colimit formula   
    \[
    \bL \mM (H)  = \underset{\phi(G/K) \to H}{\colim} \mM(G/K)
    \]
    where the colimit is over the slice category \(\phi \downarrow H\) of maps of the form 
    \(\phi(G/K) \to H\) in \(\widetilde{\mcS}^{op}\), with \(G/K \in \mcS^{op}\).
    In Lemma \ref{lem:cofinalnightmare} we show that there is a cofinal inclusion 
    \(\iota \colon \phi^{-1} (H) \to \phi \downarrow H\), where \(\phi^{-1} (H)\)
    is the subcategory of \(\mcS^{op}\) with one object \(G/H\) and morphisms 
    the endomorphisms of \(G/H\). 
    The inclusion sends 
    the object \(G/H\) to the identity map \(\phi(G/H) = H\)
    and a morphism \(f \colon G/H \to G/H \in \mcS^{op}\) to the identity morphism
    \[
    \Id \colon (\phi(G/H) = H) \to (\phi(G/H) = H).
    \]
    Thus, \(\phi^{-1} (H)\) is the one-point category with morphisms given by the group \(G/H\)
    and the cofinal functor sends all of this category to the identity, so 
    \[
    \bL \mM (H)  = \underset{\phi^{-1} (H)}{\colim} \mM(G/H) = \mM(G/H)_{G/H}.
    \]
    Similar statements hold to show that the right adjoint has the claimed form. 
    
    Now, we prove that the two adjoints of \(\underline{\phi}^*\) are naturally isomorphic. 
    Since we are working rationally, for any \( \bQ[G]\)-module \(A\) the composite 
    \(  A^G \to A \to A_G \) is an isomorphism, which is natural in \(A\).
    This natural map extends to a morphism of presheaves
    \[
    \bR \mM \longrightarrow \bL \mM
    \]
    that is an objectwise isomorphism. 
\end{proof}
   
\begin{lemma}\label{lem:cofinalnightmare}
In the setting of Proposition \ref{prop:quintlocalnoaction}, there is a cofinal inclusion
\[
\iota \colon \phi^{-1} (H) \to \phi \downarrow H 
\]
for any \(H \in \widetilde{\mcS}^{op}\).
\end{lemma}
\begin{proof}
    Following Mac Lane \cite[Section IX.3]{MacLane}, we first show that given any object of \(\phi \downarrow H\) there is a map to an object in the image of \(\iota\).
    Take \(\phi(G/K) \to H\) in the slice category, then \(K \to H\) is a map in \(\widetilde{\mcS}^{op}\).  Thus \(H\) is a subgroup of \(K\), giving the canonical projection map \(G/K \to G/H\) in \(\mcS^{op}\).
    We thus have the desired map
    \[
    (\phi(G/K) \to H) \longrightarrow (\phi(G/H) \xrightarrow{=} H).
    \]         
    Secondly, we show that given any \(\phi(G/K) \to H\) in \(\phi \downarrow H \) and maps in the slice category
    \[
    (\phi(G/H) \xrightarrow{=} H) 
    \overset{f}{\longleftarrow} (\phi(G/K) \to H) 
    \overset{g}{\longrightarrow} (\phi(G/H) \xrightarrow{=} H)
    \]
    there is a map \((\phi(G/H) \xrightarrow{=} H) \overset{h}{\longrightarrow} (\phi(G/H) \xrightarrow{=} H) \)
    with \(h \circ f= g\).

    Translating the domain terms to \(\mcS\), we have maps \(a_1 \colon G/H \to G/K\), \(gH \mapsto g\alpha_1 K\)
    and \(a_2 \colon G/H \to G/K\), \(gH \mapsto g\alpha_2 K\). 
    Define a map \(c \colon G/H \to G/H\) by \(gH \mapsto g\alpha_1^{-1} \alpha_2 H\).
    This induces the desired \(h\), proving cofinality.        
\end{proof}

\begin{example}
This cofinality result fails in the non-abelian case. 
    Consider the zig-zag in the orbit category of \(A_4\)
 \[
    A_4/\langle (12)(34)\rangle \to A_4/K \leftarrow A_4/\langle (12)(34)\rangle.
    \]  
 where one of the maps is the canonical projection and the other is given by $g\langle (12)(34)\rangle \mapsto g(123)K$.   There is no map that completes this zig-zag to a triangle.  
 
 This is the issue that underlies the failure of exactness in  Example \ref{ex:exactnessfails}---the endomorphisms at the $A_4/K$ level are ``spread out'' across the three conjugate subgroups at the $A_4/\langle(12)(34)\rangle$ level and so the colimit doesn't kill the value at $\langle (12)(34)\rangle$.
\end{example}

In light of Corollary \ref{cor:quintlocal_gldim}, Proposition \ref{prop:quintlocalnoaction} has the following implication.
\begin{corollary}\label{cor:lower_bound_coeff}
    For \(G\) an abelian group, the global dimension of \(\Coeff_{\mcS,\bQ}^{\,G}\) is bounded below by the global dimension of \(\PreSh_\bQ (\widetilde{\mcS})\).  
\end{corollary}

Thus we are interested in the global dimension of the category \(\PreSh_\bQ (\widetilde{\mcS})\) in the case of abelian groups.

\section{Incidence algebras}\label{sec:incalg}
In our search for a lower bound for rational \(\cO\)-Mackey functors, we have reduced the question to the 
global dimension of categories of rational presheaves over posets. Surprisingly, this question has been studied
extensively in the guise of global dimension of incidence algebras of Rota; see \cite{rota64}. 

\begin{definition}
    For a finite poset \(\mcP\), define \(\bQ^\mcP\) to be the rational vector space with basis given by the vertices of \(\mcP\)
    \[
    \{b_x \mid x \in \mcP\}.
    \]
    The rational \emph{incidence algebra} \(IA(\mcP)\) is the subalgebra of the matrix algebra \(\textnormal{End}(\bQ^\mcP,\bQ^\mcP)\) consisting of those endomorphisms \(f\) such that for each \(x \in \mcP\)
    \[
    f(b_x)  = \sum_{y \leq x} \alpha_y b_y.
    \]
    for some rational coefficients \(\alpha_y\), which depend upon \(x\). 
\end{definition}

For each vertex \(x \in \mcP\) there is an idempotent \(e_x\) which is the projection onto coordinate \(x\).
Any left \(IA(\mcP)\) module \(M\) provides the data 
of a vector space \(e_x M\) for each vertex \(x\) and linear maps
\[
f_{yx} \colon e_x M \to e_y M 
\]
whenever \(y < x\), such that \(f_{zy} f_{yx} = f_{zx}\) 
whenever \(z < y < x\). In fact, this information is precisely the 
data of rational presheaf on \(\mcP\), as we record in the following.

\begin{lemma}[{\cite[Section 1]{IZ90}}]\label{lem:presheafincidence}
    If \(\mcP\) is a finite poset, then there is an equivalence
    of categories between left modules over \(IA(\mcP)\) and 
    rational presheaves over \(\mcP\). 
    \[
    IA(\mcP)\dashMod \cong \PreSh_\bQ (\mcP).
    \]
\end{lemma}
For brevity, we will allow ourselves to use this equivalence
without notation and rely on context to make it clear when we 
have applied this equivalence to a presheaf or a module. 

\begin{remark}
    The algebra \(IA(\mcP)\) has another formulation via the Freyd–Mitchell Embedding Theorem. The category of rational presheaves on \(\mcP\) has a projective generator:
    \[
        P=\bigoplus_{p\in P}\Q\cdot \mcP(\mhyphen, p),
    \]
    the sum over all of the elements of \(\mcP\) of the free vector space on the representable presheaf. The category of rational presheaves is therefore equivalent to the category of right modules over the endomorphisms of \(P\). The Yoneda lemma shows that we have an isomorphism of algebras
    \[
        \operatorname{End}(P)^{\op}\cong IA(\mcP).
    \]
\end{remark}

\begin{example}
    Let \(\mcP\) be the poset \(y \to x\). 
    We may write 
    \[
    b_y = \begin{pmatrix} 1 \\ 0 \end{pmatrix}
    \quad
    b_x = \begin{pmatrix} 0 \\ 1 \end{pmatrix}  
    \]
    We may then 
    describe \(IA(\mcP)\) as two-by-two upper triangular matrices.
    Every such matrix is a linear combination of the matrices
    \[
    e_y = \begin{pmatrix} 1 & 0 \\ 0 & 0 \end{pmatrix} \quad
    e_x = \begin{pmatrix} 0 & 0 \\ 0 & 1 \end{pmatrix} \quad
    f_{yx} = \begin{pmatrix} 0 & 1 \\ 0 & 0 \end{pmatrix}.
    \]
    We see that 
    \[
    f_{yx} e_y = 0 =  e_x f_{yx}  \qquad
    f_{yx} e_x = f_{yx}  = e_y f_{yx} 
    \]
    so that if \(M\) is a left \(IA(\mcP)\)-algebra then
    \(f_{yx} e_x M\) is in the image of \(e_y\) and hence in 
    \(e_y M\). That is, we have two vector spaces 
    \(e_x M\) and \(e_y M\) and a linear map. 
    \[
    f_{yx} \colon e_x M \to e_y M. 
    \]
    This is precisely the data of a rational presheaf on this
    poset, including the contravariance.
\end{example}

There is a simple upper bound for the global dimension of the category of modules over an incidence algebra.
\begin{lemma}\label{rem:presheaves_upperBound}
The global dimension of $IA(\mcP)\dashMod\cong \PreSh_\bQ(\mcP)$ is bounded above by the height of the poset $\mcP$.
\end{lemma}
\begin{proof}The argument here is similar to the one used to bound the global dimension of $\Coeff_{S,\bQ}^{\,G}$ in Lemma \ref{lem:gdupperbound}.  In detail, if $p$ is an object of $\mcP$, then the functor $\ev_p\colon \PreSh_{\bQ}(\cP)\to \bQ\dashMod$ has a left adjoint $L_p$ defined as follows:
\begin{equation}\label{extbyzeropresh}(L_pA)(q)=\begin{cases} A &\text{if $q\leq p$}\\
0  &\text{if $q\not\leq p$} \end{cases}
\end{equation}
In particular, the analog of Lemma \ref{lem:LHs Extension by Zero} holds for these left adjoints.  Given a presheaf $\underline{V}\in \PreSh_{\bQ}(\cP)$, the map 
\[\bigoplus_p L_p\underline{V}(p)\to \underline{V}\]
is a surjection from a projective presheaf.  The formula (\ref{extbyzeropresh}) implies this map is an isomorphism at $q$ whenever $q$ is a maximal element in $\mcP$ of the support of the presheaf $\underline{V}$.  Thus the support of the kernel is strictly smaller and so by induction we find that the height of $\mcP$ is an upper bound for the global dimension of $\PreSh_{\bQ}(\mcP)$.
\end{proof}

The notion of support for a presheaf \(M\) on \(\mcP\)
corresponds precisely to the notion of those \(x \in \mcP\)
such that \(e_x M \neq 0\). Using the presheaf description we may describe the simple modules. 

\begin{lemma}
    Let \(S_x\) be the module over \(IA(\mcP)\) whose
    presheaf is supported at \(x\) where it takes value \(\bQ\).
    The simple modules of \(IA(\mcP)\) are the elements of the set
    \[
    \{ S_x \mid x \in \mcP \}.
    \]
\end{lemma}

Following \cite{IZ90}, we may reduce the question of global dimension down to the question of \(\ext\) groups between simple modules. 

\begin{proposition}
    The global dimension of \(IA(\mcP) \dashMod\) is given by the 
    largest \(n \in \bN\) such that there are \(x,y \in \mcP\)
    with \(\ext_{IA(\mcP)}^n (S_x,S_y) \neq 0\).
\end{proposition}
\begin{proof}
     The ring \(IA(\mcP)\) is a finite-dimensional algebra over a field. Hence every left module has a finite composition series with semi-simple composition factors. 
     The long exact sequence of  \(\ext\) groups 
     shows the global dimension is bounded by  \(\ext\) between semi-simple modules. The result follows. 
\end{proof}  

The values of \(\ext_{IA(\mcP)}^n (S_x,S_y)\) can be determined
by looking at the geometric realisation of the nerve of a certain sub-poset of \(\mcP\), as detailed in \cite[Proposition 1.1 and Theorem 1.2]{IZ90}. We paraphrase that result in the following, where we use \(\abs{\mhyphen}\) to denote
the geometric realisation of the nerve of a poset. 

\begin{theorem}[{\cite[Proposition 1.1 and Theorem 1.2]{IZ90}}]\label{thm:IZtheorem}
    Let \(\mcP\) be a finite poset and let \(x,y\) be elements of \(\mcP\).
\begin{enumerate}
    \item If \(y<x\) and the \emph{interval}
    \[
    I(x,y)=\{ z \in \mcP \mid y<z<x\}
    \]
    is non-empty, then 
    \[
    \ext_{IA(\mcP)}^n (S_x,S_y)
    \cong
    \widetilde{H}^{n-2} ( \abs{I(x,y)}, \bQ).    
    \]
    
    \item If \(y<x\) and \(I(x,y) \) is empty, then 
    \[
    \ext_{IA(\mcP)}^* (S_x,S_y)
    \]
    is \(\bQ\) concentrated in degree one. 
        
    \item If \(x=y\), then
    \[
    \ext_{IA(\mcP)}^* (S_x,S_x)
    \]
    is \(\bQ\) concentrated in degree zero. 

    \item If \(y \not\leq x\), then 
    \[
    \ext_{IA(\mcP)}^n (S_x,S_y) = 0.
    \]
    \end{enumerate}
\end{theorem}

\begin{example}\label{ex:cyclicpower}
    Consider the case of a cyclic group \(C_{p^n}\) of prime-power order. The poset \(\mcP=\Sub_{C_{p^n}}\) may be drawn as
\begin{center}
\begin{tikzpicture}[scale=1]
      \node (A) at (0,0) {\(C_{p^n}\)};
	 \node (B) at (2,0) {\(C_{p^{n-1}}\)};
	 \node (C) at (4,0) {\(C_{p^{n-2}}\)};
      \node (D) at (6,0) {\(\cdots\)};
      \node (E) at (8,0) {\(e\)};
      \draw[<-,] (A) to (B) ;
	 \draw[<-,] (B) to (C) ;
	 \draw[<-,] (C) to (D) ;
	 \draw[<-,] (D) to (E) ;
\end{tikzpicture}
\end{center}
    Theorem \ref{thm:IZtheorem} tells us that for \(i \leq n\) and 
    \(j+1 \leq n\)
    \[
    \begin{array}{rcl}
    \ext_{IA(\mcP)}^0 (S_{C_{p^i}},S_{C_{p^i}})
    &=& \bQ\\
    \ext_{IA(\mcP)}^1 (S_{C_{p^{j+1}}},S_{ C_{p^j}})
    &=& \bQ
    \end{array}
    \]
    and all other \(\ext\)-groups between these simple modules are zero since the geometric realisation of any non-empty interval will be a \(k\)-simplex and hence has trivial reduced cohomology.  Hence, the global dimension of \(IA(\mcP)\) is 1. We conclude that the global dimension of 
    \(\Coeff_{\bQ}^{\,C_{p^n}}\) is bounded below by 1.
    This is much lower than the upper bound of Lemma \ref{lem:gdupperbound} which is \(n\).
\end{example}

\begin{example}\label{ex:extcpq}
    Consider the case of a cyclic group \(C_{pq}\) for \(p,q\) distinct primes. The poset \(\mcP=\Sub_{C_{pq}}\) may be drawn as
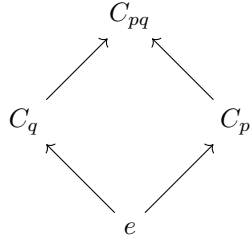
\begin{center}
\begin{tikzpicture}[scale=0.7]
      \node (A) at (0,4) {\(C_{pq}\)};
	 \node (B) at (-2,2) {\(C_{q}\)};
	 \node (C) at (2,2) {\(C_{p}\)};
      \node (D) at (0,0) {\(e\)};
      \draw[<-] (A) to (B) ;
	 \draw[<-] (A) to (C) ;
	 \draw[<-] (B) to (D) ;
	 \draw[<-] (C) to (D) ;
\end{tikzpicture}
\end{center}
    Theorem \ref{thm:IZtheorem} tells us that \(\ext\)  from a simple to itself is concentrated in degree zero, that
    \[
    \begin{array}{rcl}
    \ext_{IA(\mcP)}^1 (S_{C_{pq}},S_{C_{q}})
    &=& \bQ\\
    \ext_{IA(\mcP)}^1 (S_{C_{pq}},S_{C_{p}})
    &=& \bQ\\
    \ext_{IA(\mcP)}^1 (S_{C_{q}},S_{e})
    &=& \bQ\\
    \ext_{IA(\mcP)}^1 (S_{C_{p}},S_{e})
    &=& \bQ\\
    \ext_{IA(\mcP)}^2 (S_{C_{pq}},S_{e})
    &=& \bQ
    \end{array}
    \]
    and all other terms are zero. For the last case, 
    the interval is simply the poset \(\{ C_{q}, C_{p} \}\)  with no maps. This realises to \(S^0\) which gives a non-zero \(\ext^2\).
    
    The global dimension of \(IA(\mcP)\) is therefore 2 and the global dimension of \(\Coeff_{\bQ}^{\,C_{pq}}\) is bounded below by 2.
    This is also the upper bound of Lemma \ref{lem:gdupperbound}, so we conclude that the global dimension of \(\Coeff_{\bQ}^{\,C_{pq}}\) is 2. 
\end{example}

\begin{example}
    Consider the case of a cyclic group \(C_{pqr}\) for \(p,q,r\) distinct primes. The poset \(\mcP=\Sub_{C_{pqr}}\) may be drawn as follows.
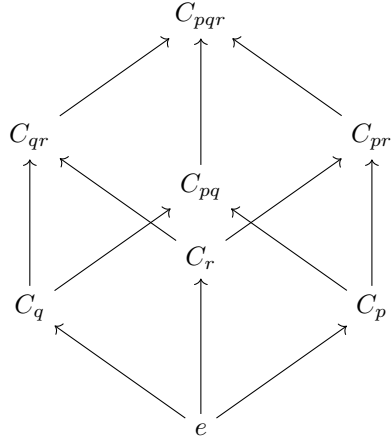
\begin{center}
\tdplotsetmaincoords{-45}{135}
\begin{tikzpicture}[tdplot_main_coords]

    \def \x {-1.6} ;
  \node (e) at (\x,\x,\x) {$C_{pqr}$};
  \node (p) at (-\x,\x,\x) {$C_{qr}$};
  \node (q) at (\x,-\x,\x) {$C_{pr}$};
  \node (r) at (\x,\x,-\x) {$C_{pq}$};
  \node (pq) at (-\x,-\x,\x) {$C_{r}$};
  \node (pr) at (-\x,\x,-\x) {$C_{q}$};
  \node (qr) at (\x,-\x,-\x) {$C_{p}$};
  \node (pqr) at (-\x,-\x,-\x) {$e$};

      \draw[<-] (e) to (p) ;
	 \draw[<-] (e) to (q) ;
	 \draw[<-] (e) to (r) ;
      \draw[<-] (p) to (pq) ;
	 \draw[<-] (p) to (pr) ;
	 \draw[<-] (q) to (pq) ;
      \draw[<-] (q) to (qr) ;
	 \draw[<-] (r) to (pr) ;
	 \draw[<-] (r) to (qr) ;
      \draw[<-] (pq) to (pqr) ;
	 \draw[<-] (pr) to (pqr) ;
	 \draw[<-] (qr) to (pqr) ;
\end{tikzpicture}
\end{center}
    Looking at the diagram we see that the most interesting interval is that from \(C_{pqr}\) to \(e\), \(I(C_{pqr},e)\). 
    This interval consists of all non-trivial proper subgroups and its realisation is the circle, hence 
    \[
    \begin{array}{rcl}
    \ext_{IA(\mcP)}^3 (S_{C_{pqr}},S_{e})
    &=& \bQ.\\
    \end{array}
    \]
    One may check that all other non-zero \(\ext\) groups are in lower degrees.
    Hence, the global dimension of \(IA(\mcP)\) is 3, which is also the upper bound of Lemma \ref{lem:gdupperbound}. We conclude that the global dimension of \(\Coeff_{\bQ}^{\,C_{pqr}}\) is 3.
\end{example}

We learnt of the following from \cite{IZ90}, which provides an upper bound for many simple posets. 

\begin{theorem}[{\cite[Proposition 3.4]{AG87}}]\label{thm:planaristwo}
    If \(\mcP\) is a finite planar poset, then the global dimension of \(IA(\mcP)\) is at most 2. 
\end{theorem}

Theorem \ref{thm:IZtheorem} gives us an effective method for calculating the 
global dimension of presheaves on a poset. 
We finish this section by connecting this result to coefficient systems. 
That is accomplished by the following two theorems, which show that adding in group actions does not change the global dimension. 

\begin{theorem}\label{thm:equiv_pres=pres}
    For \(G\) a finite group, the global dimension of \(G\)-objects in rational presheaves on a poset \(\mcP\) equals the global dimension of rational presheaves on \(\mcP\).
\end{theorem}

\begin{proof}
For \(\cC\) a category, let \(\cC[G]\) denote the category of objects in \(\cC\) with an action of \(G\)
and morphisms those maps which commute with the \(G\)-actions.     
Consider the category of presheaves on a poset \(\mcP\) with values in \(\bQ[G]\)-modules. This category is isomorphic to the category of \(G\)-objects in rational presheaves on a poset \(\mcP\). By Lemma \ref{lem:presheafincidence}, that category is, in turn, isomorphic to the category of \(G\)-objects in \(IA(\mcP)\dashMod\), as displayed below 
\[
\PreSh_{\bQ[G]}(\mcP) \cong
\PreSh_\bQ (\mcP)[G] 
\cong \left(IA(\mcP)\dashMod \right)[G].
\]
For a finite group \(G\), the category of \(\bQ[G]\)-modules is semi-simple, which implies the following key facts for \(G\)-objects in the category of modules over the rational incidence algebra \(IA(\mcP)\).  
\begin{enumerate}
    \item A simple object has the form \(S_H \otimes V\) where \(V\) is an irreducible representation of \(G\).
    \item If \(V,W\) are non-isomorphic irreducible representations, then 
    \[
    \ext_{\left( IA(\mcP)\dashMod \right)[G]}^n( S_H \otimes V, S_K \otimes W) =0.
    \]
    \item If \(V,W\) are isomorphic irreducible representations, then 
    \[
    \ext_{\left( IA(\mcP)\dashMod \right)[G]}^n( S_H \otimes V, S_K \otimes W)\neq 0 \text{\quad if and only if \quad} 
     \ext_{IA(\mcP)}^n( S_H , S_K ) \neq 0. \qedhere
    \]
\end{enumerate}
\end{proof}

When \(G\) is abelian, a \(G\)-coefficient system may be thought of as an object of \( \PreSh_\bQ (\mcP)[G] \) for \(\mcP=\Sub_G\)  satisfying a Weyl condition (that is, its value is \(H\)-fixed at place \(H\)). Therefore we can consider an inclusion functor 
\[\Coeff_\bQ^{\,G} \longrightarrow \PreSh_\bQ(\Sub_G)[G] ,\]
which has naturally isomorphic left and right adjoints given by \(H\)-orbits and \(H\)-fixed points at place \(H\), respectively. 
Combining this observation with Proposition \ref{prop:quintlocalnoaction},  we can sandwich \(G\)-coefficient systems between two presheaf categories of the same global dimension.

\begin{theorem}[Abelian Sandwich]\label{thm:sandwich}
    For \(G\) a finite abelian group, there are quintessential subcategories 
    \[
    \PreSh_{\bQ}(\Sub_G) \longrightarrow     \Coeff_\bQ^{\,G}
    \longrightarrow    \PreSh_\bQ (\Sub_G)[G]. 
    \]  
    Thus the global dimension of rational \(G\)-coefficient systems is the same as the global dimension of rational presheaves on \(\Sub_G\).
\end{theorem}

\section{Global dimension of \texorpdfstring{\( \OMackey_\bQ^G\)}{O-MackeyQ(G)} for \texorpdfstring{\(G\)}{G} abelian and disk-like \texorpdfstring{\(\cO\)}{O}}\label{sec:gldimmackey}
Our main method for calculating global dimension comes from
translating to presheaves by the Abelian Sandwich Theorem, Theorem \ref{thm:sandwich}, and 
calculating homology of a geometric realisation of an interval in \(\Sub_G\)
using Theorem \ref{thm:IZtheorem}. 

Using the group structure of \(G\) we can often reduce these calculations to smaller posets.

\begin{lemma}\label{lem:interval_reductions}
For \(G\) a group with subgroups \(K \subseteq H\) and \(L \trianglelefteq H\), 
there are isomorphisms  
    \[\begin{array}{rcl}
    \ext^n_{IA(\Sub_G)}(S_H, S_K) & \cong & \ext^n_{IA(\Sub_H)}(S_H, S_K)
    \\
    \ext^n_{IA(\Sub_H)}(S_H, S_L) & \cong & \ext^n_{IA(\Sub_{H/L})}(S_{H/L}, S_{L/L}).
    \end{array}
    \]    
\end{lemma}
\begin{proof}
    The first isomorphism holds as the interval in \(\Sub_G\) between \(H\) and \(K\) is 
    isomorphic to the interval in \(\Sub_H\) between \(H\) and \(K\). 
    The second is that the interval in \(\Sub_H\) between \(H\) and \(L\)
    is isomorphic to the interval in \(\Sub_{H/L}\) between \(H/L\) and \(L/L\), by the correspondence
    theorem for quotient groups.
\end{proof}

To understand geometric realisations of intervals we begin by recalling a result from simplicial homotopy theory; see for example \cite[1.3]{QUILLEN_posets}.

\begin{proposition}\label{prop:simpl_homot}
    If \(f_0,f_1 \colon \mcP \longrightarrow \mcQ\) are maps of posets such that for all \(x\in \mcP\), \(f_0(x)\leq f_1(x)\), then \(\abs{f_0}\) is homotopic to \(\abs{f_1}\), where \(\abs{\cdot}\) denotes geometric realisation.
\end{proposition}
\begin{proof}
    We can define a map of posets \(h\colon \{0 \rightarrow 1\} \times \mcP \longrightarrow \mcQ\) by \(h(i,x)= f_i(x)\). After geometric realisation this gives a map \[|h|\colon[0,1]\times |\mcP| \longrightarrow |\mcQ|,\]
    which is a homotopy from \(\abs{f_0}\) to \(\abs{f_1}\).
\end{proof}

We will use this result to understand when the geometric realisation of the interval from Theorem \ref{thm:IZtheorem} in the case \(\mcP =\Sub_G\) is contractible. Recall that the \emph{Frattini subgroup} \(\Phi(H)\) of a group \(H\) is the intersection of all the maximal subgroups of \(H\).

\begin{theorem}\label{thm:Frattini_interval}
    Consider the poset \(\Sub_G\) and suppose \(K< H\). If \(I(H,K)\) is non-empty and \(\Phi(H)\) is not a subgroup of \(K\), then \(\abs{I(H,K)}\) is contractible. 
\end{theorem}

\begin{proof}
    We will construct three maps of posets 
    \[
    f_0,f_1,f_2\colon I(H,K) \longrightarrow I(H,K) 
    \]
    such that for all \(L\in I(H,K)\), \(f_0(L)\leq f_1(L)\geq f_2(L).\)
Set \(f_0=\Id_{I(H,K)}\) and \(f_2\) the constant function at the subgroup \(\Phi(H) K\). 

Since \(\Phi H\not\leq K\), \(\Phi(H) K\) strictly contains $K$. If \(K'\) is a maximal proper subgroup of \(H\) that contains \(K\), then by construction, it necessarily also contains \(\Phi(H)\), and hence also \(\Phi(H)K\). Thus \(\Phi(H)K\) is in the interval \(I(H,K)\). Define 
\[
    f_1(L)=(\Phi(H)K)\cup L.
\]
Notice that this is indeed a map of posets \(I(H,K) \to I(H,K)\), since \(\Phi(H) K\) is in all maximal proper subgroups of \(H\) that contain \(K\). 

Observe that by definition, \(\Phi(H) K\leq (\Phi(H) K)\cup L\) for all $L\in I(H,K)$.  Hence $f_0\leq f_1\geq f_2$. Now we use Proposition \ref{prop:simpl_homot} twice to get a homotopy from the identity \(\abs{f_0}\) on \(\abs{I(H,K)}\) to the constant map \(\abs{f_2}\). Thus \(\abs{I(H,K)}\) is contractible.
\end{proof}

Lemma \ref{lem:interval_reductions} and Theorem \ref{thm:Frattini_interval} imply that the only intervals that contribute to the global dimension of rational presheaves on \(\Sub_G\) are of the form \(I(H,K)\), where \(\Phi H \leq K < H\).

\begin{lemma}\label{lem:interval_sphere}
    If \(G=(C_p)^{\times k}\), then the global dimension of rational presheaves on \(\Sub_G\) is \(k\) and it is realised by \(I(G, e)\). 
\end{lemma}
\begin{proof}
  The height of the poset \(\Sub_G\) is \(k\), so the global dimension of rational presheaves on \(\Sub_G\) is less than or equal to \(k\). We will show it is \(k\). 

Consider the inclusion \(i\colon (\Sub_{C_p})^{\times k} \rightarrow \Sub_G\). This has a retraction: if we let \(\vec{e}_i\in C_p^{\times k}\cong \F_p^{k}\) be the $i$\textsuperscript{th} standard basis vector, the retraction \(r\) is given by 
\[
L\mapsto \max\{ K \leq L\mid K\in (\Sub_{C_p})^{\times k}\}=\Span\{ \vec{e}_i\mid \vec{e}_i\in L\}.
\]
In other words we assign to \(L\) the maximal subgroup of \(L\) that is a product of subgroups of the $C_p$ factors.  At the level of posets we have an adjunction where \(i\) is the left adjoint of \(r\).

The maps \(i\) and \(r\) induce restriction functors 
\[
\begin{array}{rcl}
i^* \colon \PreSh_\bQ (\Sub_G) 
& \longrightarrow &
\PreSh_\bQ ((\Sub_{C_p})^{\times k}) \\
r^* \colon \PreSh_\bQ ((\Sub_{C_p})^{\times k}) 
& \longrightarrow &
\PreSh_\bQ (\Sub_G) 
\end{array}
\]
that form an adjoint pair. 

The functor \(r^*\) is fully faithful because
\(i^* \circ r^* = \Id\) and it is exact as it also has a right adjoint given by right Kan extension. Thus by the right handed version of 
Lemma \ref{lem:adjunction_and_global_dim},
the global dimension of 
\(\PreSh_\bQ ((\Sub_{C_p})^{\times k}) \) is less or equal to  the global dimension of 
\(\PreSh_\bQ (\Sub_G) \). 

We will now show that the global dimension of 
\(\PreSh_\bQ ((\Sub_{C_p})^{\times k}) \) is \(k\).
Since the interval between the maximal and minimal elements in the poset \((\Sub_{C_p})^{\times k}\) is isomorphic to the poset of non-empty, proper subsets of \(k\)-element set, the geometric realisation of that interval is homotopy equivalent to the sphere of dimension \(k-2\). Thus by Theorem \ref{thm:IZtheorem} the global dimension of rational presheaves on this poset is at least \(k\). Notice that the global dimension also at most \(k\), since that is the height of the poset \((\Sub_{C_p})^{\times k}\).  Hence the global dimension is \(k\).

Finally, the global dimension is realised by the interval \(I(G,e)\) in \(\Sub_G\) contributing a non-trivial Ext group in dimension \(k\). 
By Lemma \ref{lem:interval_reductions}, any other interval \(I(H,K)\) would be isomorphic to an interval \(I(H/K, e)\) in a poset \(\Sub_{H/K}\) of a strictly lower height than \(k\) and thus can contribute a non-trivial Ext group only up to the dimension equal to the height of the poset \(\Sub_{H/K}\); see Remark \ref{rem:presheaves_upperBound}. 
\end{proof}

\begin{corollary}\label{cor:p-group}
    For an abelian \(p\)-group \(G\) the global dimension of rational presheaves on \(\Sub_G\) is the number of prime-power cyclic factors of \(G\).
\end{corollary}
\begin{proof}
   Let \(G\cong C_{p^{a_1}}\times C_{p^{a_2}}\times \dots \times C_{p^{a_n}} \).  Then any subgroup \(H \leq G\) 
   is isomorphic to a group of the form \(C_{p^{b_1}}\times C_{p^{b_2}}\times \dots \times C_{p^{b_k}} \) where \(k\leq n\). 
By Miller \cite{Miller15}, the Frattini subgroup of \(H\) is 
\[\Phi (C_{p^{b_1}}\times C_{p^{b_2}}\times \dots \times C_{p^{b_k}} )= C_{p^{b_1-1}}\times C_{p^{b_2-1}}\times \dots \times C_{p^{b_k-1}}.\]
Thus the interval \(I(H,\Phi H)\) in \(\Sub_G\) is of the form 
\[I(C_{p^{b_1}}\times C_{p^{b_2}}\times \dots \times C_{p^{b_k}},C_{p^{b_1-1}}\times C_{p^{b_2-1}}\times \dots \times C_{p^{b_k-1}})\cong I((C_p)^{\times k}, e ),\]
where the last interval is taken in the poset \(\Sub_{(C_p)^{\times k}}\).

By Lemmas \ref{lem:interval_reductions} and \ref{lem:interval_sphere},  
the interval \(I(H, \Phi H)\), which is isomorphic to \(I((C_p)^{\times k}, e)\) in the poset \(\Sub_{(C_p)^{\times k}}\), 
realises the non-trivial Ext group in dimension \(k\).  Thus the global dimension of rational presheaves on \(\Sub_G\) is realised by the interval where the dimension of the non-trivial Ext is the highest, which is for \(H=G\). The global dimension is then equal to \(n\), which is the number of prime-power cyclic factors of \(G\).
\end{proof}

\begin{theorem}\label{thm:gldim_abelian_group}
    For \(G\) a finite abelian group, the global dimension of rational presheaves on \(\Sub_G\) is the number of prime-power cyclic factors of \(G\).
\end{theorem} 
\begin{proof}
The group \(G\) is a finite product of its maximal \(p\)-groups \(G_p\)
\[
G\cong \prod_p G_p.
\]
We use a series of simplifications so that we may work one prime at a time. 

Firstly, if \(H,K\) are groups of coprime order then \(\Sub_{H \times K} \cong \Sub_H \times \Sub_K \). Secondly, if \(\mcP\) and \(\mcQ\) are locally finite posets, then \(IA(\mcP) \otimes_{\bQ} IA(\mcQ) \cong IA(\mcP \times \mcQ) \), see \cite[Proposition 2.1.12]{SpiegelODonnell}. 
Thirdly, if \(A,B\) are finite dimensional algebras over \(\bQ\), then 
\[
\gldim(A \otimes_\bQ B)=\gldim(A)+\gldim(B);
\] 
see \cite{Auslander}.

By Corollary \ref{cor:p-group} the global dimension for the maximal \(p\)-subgroup \(G_p \leq G\) is the number of prime-power cyclic factors in \(G_p\), which is the number of \(p\)-cyclic factors in \(G\). 
Thus, the global dimension of rational presheaves on \(\Sub_G\) is the sum of the number of \(p\)-cyclic factors for all prime \(p|\abs{G}\), which is the number of cyclic factors of \(G\). 
\end{proof}

\begin{corollary}\label{cor:abelianfrattini}
    For \(G\) a finite abelian group, the global dimension of rational presheaves on \(\Sub_G\) is realised by the interval \(I(G, \Phi G)\). That is, the global dimension is given by the largest \(n\) such that 
    \[
    \ext^n_{IA(\Sub_G)}(S_G, S_{\Phi G}) \neq 0.
    \]    
\end{corollary}
\begin{proof}
    Miller \cite{Miller15} proves that \(\Phi(H \times K) \cong \Phi(H) \times \Phi(K)\) for any finite groups \(H,K\). 
    The result follows from that fact and the proofs of Corollary \ref{cor:p-group} and Theorem \ref{thm:gldim_abelian_group}.
\end{proof}

In view of the Abelian Sandwich Theorem, Theorem \ref{thm:sandwich}, and 
Theorem \ref{thm:gldim_abelian_group} we obtain the following key result.

\begin{corollary}\label{cor:gldim_coeffsyst}
 For \(G\) an abelian group, the global dimension of rational \(G\)-coefficient systems is the number of prime-power cyclic factors of \(G\).  
\end{corollary}

Now suppose \(\cO\) is a disk-like transfer system for an abelian group \(G\). We want to compute the global dimension of the category of rational \(\cO\)-Mackey functors.  There are two complications in this case which did not exist for rational presheaves on \(\Sub_H\), both related to the poset structure of \(\Sub_{\langle H \rangle}^{\cO}\). For an abelian group \(G\), there is one maximal element in \(\Sub_{\langle H \rangle}^{\cO}\), namely \(H\), but this sub-poset is no longer closed under meets and may have many incomparable minimal elements. This has two consequences. Firstly, we cannot separate different primes for \(H\) since \(\Sub_{\langle H \rangle}^{\cO}\) might not be a product of posets for these primes). Secondly our calculations of global dimension in this section depended on \(G\) and the Frattini subgroup of \(G\). In the case of \(\Sub_{\langle H \rangle}^{\cO}\), the Frattini subgroup of \(H\) may or may not be in \(\Sub_{\langle H \rangle}^{\cO}\).

\begin{definition}\label{def:n_h}
    Suppose \([H]^\cO\) is the inseparability class of \(H\) for a disk-like transfer system \(\cO\) for an abelian group \(G\). Define \(\dim ([H]^\cO)\) to be the maximum number of the prime-power  cyclic factors of \(L/K\) where \(L\geq K\) and both \(L,K\) run over the subgroups in \(\Sub_{\langle H \rangle}^\cO\).
\end{definition}
Notice that the set over which we take the maximum can be chosen to be much smaller. In fact, \(\dim ([H]^\cO)\) is the maximum  number of prime-power cyclic factors of \(H/K\) where \(K\) runs over the minimal elements in \(\Sub_{\langle H \rangle}^\cO\).
However, the given definition of \(\dim ([H]^\cO)\) will make it easy to deduce monotonicity of the global dimensions in the poset of disk-like transfer systems in Section \ref{sec:monotonicity}.

\begin{proposition}\label{prop:gl_dim_cloud}
    Let \(G\) be an abelian group and \(\cO\) a disk-like transfer system for \(G\). Then the global dimension of rational presheaves on \(\Sub_{\langle H \rangle}^{\cO}\) is \(\dim ([H]^\cO)\).
\end{proposition}
\begin{proof}
    Using the incidence algebra approach for \(\mcP=\Sub_{\langle H \rangle}^\cO\) we want to compute \(\ext^n_{IA(\mcP)}(S_L, S_K)\) for \(K\leq L \in \Sub_{\langle H \rangle}^\cO\). 
    The closure property of $[H]^{\mcO}$ recalled in Remark \ref{rmk:relfamily} implies that the interval $I_{\mcP}(L,K)$ between $L$ and $K$ in $\mcP$ is the same as the interval $I_{\Sub_H}(L,K)$ between $L$ and $K$ in $\Sub_H$. Lemma \ref{lem:interval_reductions} then implies 
    \[
    \ext^n_{IA(\mcP)}(S_L, S_K)\cong \ext^n_{IA(\Sub_H)}(S_L, S_K) \cong \ext^n_{IA(\Sub_{H/K})}(S_{L/K}, S_{K/K}).
    \] 

For a given \(K \in \Sub_{\langle H \rangle}^\cO\) and any \(L\geq K\) the last Ext group is calculated for the incidence algebra over the poset \(\Sub_{H/K}\). Notice that if \(K\) is not a minimal subgroup in \(\Sub_{\langle H \rangle}^\cO\) then there exists a minimal subgroup \(K' \in \Sub_{\langle H \rangle}^\cO\) such that \(K'\leq K\) and 
\[
\ext^n_{IA(\Sub_{H/K})}(S_{L/K}, S_{K/K})\cong \ext^n_{IA(\Sub_{H/K'})}(S_{L/K'}, S_{K/K'})
\] 
can be calculated in \(\Sub_{H/K'}\). Thus we are back to the case from Theorem \ref{thm:gldim_abelian_group}, giving us that the global dimension is the number of prime-power cyclic factors of \(H/K'\). 
Since we need to vary \(K\) in the calculation above, we need to take the maximum over all minimal \(K'\in \Sub_{\langle H \rangle}^\cO\) of the number of prime-power cyclic factors of \(H/K'\), which is \(\dim ([H]^\cO)\).
\end{proof}

\begin{theorem}\label{thm:one_cloud}
    Let \(G\) be an abelian group and \(\cO\) a disk-like transfer system for \(G\). Then the global dimension of the category \(\Coeff_{\langle H\rangle^{\cO},\bQ}^{\,G}\) is \(\dim ([H]^\cO)\). 
\end{theorem}
\begin{proof}
    We again sandwich the category \(\Coeff_{\langle H\rangle^{\cO},\bQ}^{\,G}\) between rational presheaves on \(\Sub_{\langle H \rangle}^{\cO}\) and rational \(G\)-presheaves on \(\Sub_{\langle H \rangle}^{\cO}\) using a pair of quintessential localisations obtained from Theorem \ref{thm:sandwich} by restricting to the poset \(\Sub_{\langle H \rangle}^{\cO}\). Both of the presheaf categories have the same global dimension by Theorem \ref{thm:equiv_pres=pres}. By Proposition  \ref{prop:gl_dim_cloud} this dimension is \(\dim ([H]^\cO)\).   
\end{proof}

\begin{theorem}\label{thm:main_thm_disklike}
    Let \(G\) be an abelian group and \(\cO\) a disk-like transfer system for \(G\). The global dimension of the category of rational \(\cO\)-Mackey functors is the maximum over all subgroups \(H \in {\SubOG}\) of \(\dim ([H]^\cO)\). 
\end{theorem}

\begin{proof}
In case of an abelian group \(G\) and a disk-like transfer system \(\cO\), 
Corollary \ref{cor:abeliansplittingIncompleteMackey} gives an equivalence of categories
     \[
        \OMackey_\bQ^G \cong \prod_{H\in {\SubOG}} \Coeff_{\langle H \rangle^{\cO},\bQ}^{\,G}.
    \]
By Theorem \ref{thm:one_cloud} the global dimension of \(\Coeff_{\langle H \rangle^{\cO},\bQ}^{\,G}\) is \(\dim ([H]^\cO)\); thus the global dimension of \(\OMackey_\bQ^G\) is the maximum over all \(H\in {\SubOG}\) of \(\dim ([H]^\cO)\).
\end{proof}

Notice that if \(\cO\) is the trivial transfer system Theorem \ref{thm:main_thm_disklike} recovers Theorem \ref{thm:sandwich} and if \(\cO\) is the complete transfer system Theorem \ref{thm:main_thm_disklike} recovers the abelian case of Remark \ref{rmk:completecasebound}.

\begin{corollary}\label{cor:dimzerocase}
    Let \(G\) be an abelian group and \(\cO\) a disk-like transfer system for \(G\). 
    The global dimension of \(\OMackey_\bQ^G\) is zero if and only if 
    \(\cO\) is the complete transfer system.
\end{corollary}
\begin{proof}
    If \(\cO\) is the complete transfer system, then the posets
    of the split pieces of Theorem \ref{thm:splittingIncompleteMackey} 
    are all singletons and so the global dimension is zero.

    For the converse, if the global dimension is zero, then \(\dim ([H]^\cO)=0\) 
    for all inseparability classes \([H]^\cO\). This implies that each inseparability class 
    is just a singleton, hence \(\cO\) contains \(H \to G\) for each subgroup \(H\). 
    This implies \(\cO\) is the complete transfer system. 
\end{proof}

\section{Monotonicity of the global dimension over the poset of transfer systems}\label{sec:monotonicity}
Results above imply that the global dimension is a monotone function on disk-like transfer systems for an abelian group \(G\) with the minimal value (zero) for the complete transfer system for \(G\) and the maximal value equal to the number of prime-power cyclic factors of \(G\) for the trivial transfer system for \(G\).

To be more precise, let \(\cO_1 \le \cO_2\) be an inclusion of disk-like transfer systems for an abelian group \(G\). 
Proposition 3.16 and Definition 3.22 from \cite{paper1} imply that the partition of \(\Sub_G\) into inseparability classes given by \(\cO_2\) is finer than the partition given by \(\cO_1\). In other words each of the posets \(\Sub_{\langle H \rangle}^{\cO_2}\) is contained in one of the posets \(\Sub_{\langle H' \rangle}^{\cO_1}\), for some \(H'\in \Sub^{\cO_1}_G\), and when we have such a containment, Definition \ref{def:n_h}  implies that 
\[
\dim\big([H]^{\cO_2}\big)\leq \dim\big([H']^{\cO_1}\big).
\]
Hence the maximum of $\dim([H]^{\cO_2})$ over all subgroups \(H \in \Sub^{\cO_2}_G\) is therefore less or equal than the maximum of \(\dim ( [H]^{\cO_1} )\) over all subgroups \(H \in \Sub^{\cO_1}_G\).
By Theorem \ref{thm:main_thm_disklike}, the global dimension of the category of rational \(\cO\)-Mackey functors is the maximum over all subgroups \(H \in {\SubOG}\) of \(\dim ([H]^\cO)\), thus we obtain the following result. 

\begin{theorem}\label{thm:monotonicity}
Suppose \(\cO_1 \le \cO_2\) is an inclusion of disk-like transfer systems for an abelian group \(G\). Then 
\[\gldim(\cO_2\mhyphen\Mackey_\bQ^G) \leq \gldim (\cO_1\mhyphen\Mackey_\bQ^G).\]
\end{theorem}

The rest of this section gives an alternative proof of this result, 
using our understanding of inclusions of posets induced by inseparability classes and the adjunctions they induce on rational presheaves.  This proof does not rely on calculating the values of the global dimensions of the categories in question.  It provides a more structural understanding of why the monotonicity result holds and we hope to use it in extending our work beyond disk-like transfer systems and abelian groups.

\begin{proposition}\label{prop:right_leftKan_zero}
    Consider a full inclusion of posets \(i\colon\cP\hookrightarrow \mcQ\) and its complement \(j\colon \mcQ\setminus\cP\hookrightarrow \mcQ\). If there are no maps from any object in \(\cP\) to any object in \(\mcQ\setminus \cP\), then
    \begin{enumerate}
        \item the right adjoint \(i_\ast\) to the restriction functor
        \[
            i^\ast \colon \PreSh_\bQ (\mcQ) \to \PreSh_\bQ (\cP)
        \]
        is given by extension by zero, and
        \item the left adjoint \(j_{!}\) to the restriction functor
        \[
            j^\ast\colon \PreSh_{\bQ}(\mcQ)\to\PreSh_{\bQ}(\mcQ\setminus\cP)
        \]
        is given by extension by zero.
    \end{enumerate}
\end{proposition}
\begin{proof}
The right adjoint to the restriction functor is the right Kan extension. In this setting, since we have the fully faithful functor \(i\), the value of the right Kan extension for every \(p\in \cP\) is unchanged, and the value for any \(q\in \mcQ\setminus \cP\) can be computed as a limit over a slice category which is empty, by assumption. The limit gives the terminal object, which is zero, thus the right adjoint \(i_\ast\) to the restriction functor \(i^\ast\) is extension by zero.
A dual argument gives the second case. 
\end{proof}

An immediate consequence of the definition of the partial order on transfer systems is that the assignment
\[
    \cO\mapsto \Sub^{\cO}_{G}
\]
is monotonic on disk-like transfer systems. Because of this, for a general inclusion \(\cO_1\le\cO_2\) of disk-like transfer systems, the partitioning of \(\Sub_G\) into inseparability classes for \(\cO_2\) refines the partitioning into inseparability classes for \(\cO_1\):
\[
    [H]^{\cO_1}=\coprod_{K\in [H]^{\cO_1}\cap \Sub^{\cO_2}_G} [K]^{\cO_2}=[H]^{\cO_2}\amalg\dotsb.
\]
Applying this to the sub-posets of the subgroups of \(G\), we have a decomposition
\[
\Sub_{\langle H\rangle}^{\cO_1}=\Sub_{\langle H\rangle}^{\cO_2}\amalg \big(\Sub_{\langle H\rangle}^{\cO_1}\setminus\Sub_{\langle H\rangle}^{\cO_2}\big).
\]
Since for any \(\cO\), \(H\) is the maximal element in both posets \(\Sub_{\langle H\rangle}^{\cO}\), the closure property in Remark \ref{rmk:relfamily} implies that the inclusion
\[
i\colon \Sub_{\langle H \rangle}^{\cO_2} \rightarrow \Sub_{\langle H \rangle}^{\cO_1}
\]
satisfies the assumptions of Part 1 of Proposition \ref{prop:right_leftKan_zero}: if $K\in[H]^{\mcO_2}$ and $L\in [H]^{\mcO_1}$  satisfy $K\le L\le H$, then $L\in[H]^{\mcO_2}$. Thus the right adjoint \(i_{\ast}\) to the  restriction of rational presheaves
\[
i^* \colon \PreSh_\bQ (\Sub_{\langle H \rangle}^{\cO_1}) 
\to
\PreSh_\bQ (\Sub_{\langle H \rangle}^{\cO_2})
\]
is the extension by zero, and so is the left adjoint \(j_{!}\) to the restriction
\[
j^* \colon \PreSh_\bQ (\Sub_{\langle H \rangle}^{\cO_1}) 
\to
\PreSh_\bQ \big(\Sub_{\langle H \rangle}^{\cO_1}\setminus\Sub_{\langle H \rangle}^{\cO_2}\big).
\]
This has two consequences for the global dimension. First, \(i_{\ast}\) is an exact fully-faithful functor with exact left adjoint, hence by the right-handed version of  
Lemma \ref{lem:adjunction_and_global_dim},  the global dimension of \(\PreSh_\bQ (\Sub_{\langle H \rangle}^{\cO_2}) \) cannot exceed the global dimension of \(\PreSh_\bQ (\Sub_{\langle H \rangle}^{\cO_1}) \). Second, \(j_{!}\) is an exact fully-faithful functor with an exact right adjoint, and hence the global dimension of 
\(\PreSh_\bQ \big(\Sub_{\langle H \rangle}^{\cO_1}\setminus\Sub_{\langle H \rangle}^{\cO_2}\big) \) cannot exceed the global dimension of 
\(\PreSh_\bQ (\Sub_{\langle H \rangle}^{\cO_1}) \). 

Combining the two inequalities above we get that the global dimension of 
\(\PreSh_\bQ (\Sub_{\langle H \rangle}^{\cO_2})\) and the global dimension of \(\PreSh_\bQ \big(\Sub_{\langle H \rangle}^{\cO_1}\setminus\Sub_{\langle H \rangle}^{\cO_2}\big)\) are both less than or equal to the global dimension of \(\PreSh_\bQ (\Sub_{\langle H \rangle}^{\cO_1}) \).

Now we will use the fact that \(\cO_2\) gave a finer partition than \(\cO_1\) of subgroups of \(G\). Thus for a maximal subgroup \(K\in \Sub_{\langle H \rangle}^{\cO_1}\setminus\Sub_{\langle H \rangle}^{\cO_2}\) we have an inclusion of posets \(\Sub_{\langle K \rangle}^{\cO_2}\rightarrow \Sub_{\langle H \rangle}^{\cO_1}\setminus\Sub_{\langle H \rangle}^{\cO_2}\) satisfying again the assumptions of Part 1 of Proposition \ref{prop:right_leftKan_zero}. Notice that its complement in \(\Sub_{\langle H \rangle}^{\cO_1}\setminus\Sub_{\langle H \rangle}^{\cO_2}\) satisfies again the conditions of Part 2 of Proposition \ref{prop:right_leftKan_zero}. We can proceed by induction, since the partition of each inseparability class \([H]^{\cO_1}\) into inseparability classes for \(\cO_2\) is finite. We therefore obtain a different proof of Theorem \ref{thm:monotonicity}, one that does not depend on knowing the global dimensions.

\section{Three conjectures}\label{sec:conjectures}
Questions on the homology and homotopy of intervals of \(\Sub_G\) can be quite difficult to answer. For example, there is a conjecture by Shareshian from 2003 that posits that  every open interval in the lattice of subgroups of a finite group has the homotopy type of a wedge of spheres. This is known for solvable groups, but seems to be open in general; see Shareshian \cite{Shareshian}. In this spirit, we state three conjectures on questions on posets and global dimension.

Theorem \ref{thm:monotonicity} gives monotonicity of global dimension over disk-like transfer systems for an abelian group \(G\). We conjecture that the global dimension is a monotone function on the poset of all transfer systems for any finite group \(G\).

\begin{conjecture}\label{conj:monotone}
    The global dimension of the category of rational \(\cO\)-Mackey functors is a monotone function over a poset of all transfer systems \(\cO\) for any finite group \(G\), with the minimal value (zero) attained only at the complete transfer system for \(G\). 
\end{conjecture}

In general, we expect the global dimensions of the incidence algebras for the posets \(\Sub_G\) and \({\Sub_G}/G\) to differ, but we do not know of an example. Furthermore, Example \ref{ex:exactnessfails} shows that the non-abelian case is substantially more complicated. These points lead us to make the following conjecture. 

\begin{conjecture}\label{conj:2}
    There exists a group \(G\) where the global dimension of the rational incidence algebra for \({\Sub_G}/G\) and the global dimension of the category of rational coefficient systems for \(G\) are different.
\end{conjecture}

Further to the question of the global dimension of a category, one may ask for a pair of objects
where this dimension is realised. In the case of incidence algebras over \(\Sub_G\), this would mean identifying a 
pair of subgroups \(K,H\) such that \(\ext^n(S_H, S_K) \neq 0\) for \(n\) the global dimension. 

\begin{conjecture}[Frattini conjecture]\label{conj:frattini}
    Let \(G\) be a finite group. 
    The global dimension of \(IA(\Sub_G)\) is 
    given by the largest \(n\) such that 
    there is a subgroup \(H\) with 
    \[
    \ext^n_{IA(\Sub_G)}(S_H, S_{\Phi H}) \neq 0.
    \]    
\end{conjecture}

\begin{example}
We illustrate why the Frattini conjecture involves maximizing over all subgroups \(H\) of \(G\), rather than 
just looking at \(G\) and its Frattini subgroup, as we do in the abelian case. 

    Let \(F_5= C_5 \rtimes C_4 = \aut(D_5)\). 
    The action of \(C_4\) on \(C_5\) is defined by raising elements to their third power, 
    \(x \mapsto x^3\).
    We draw \(\Sub_{F_5}\) and \(\Sub_{F_5}/F_5\).
    We indicate different copies of \(C_2\) and \(C_4\) with superscipts. 
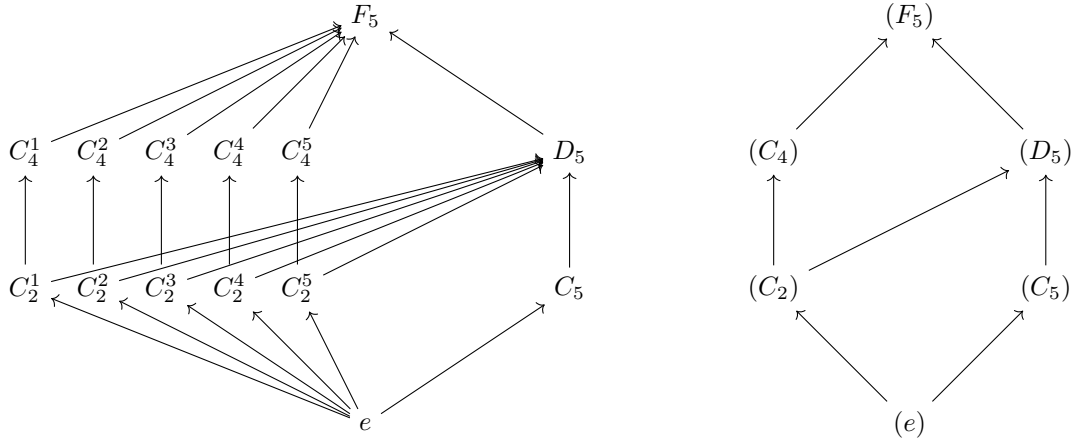
\begin{center}
\begin{tikzpicture}[scale=0.9]
     \def \xshift {0} ;
     \def \yshift {0};
      \node (f5) at (0+\xshift,6+\yshift) {$F_5$};

      \node (c4a) at (-5+\xshift,4+\yshift) {$C_4^1$};
      \node (c4b) at (-4+\xshift,4+\yshift) {$C_4^2$};
      \node (c4c) at (-3+\xshift,4+\yshift) {$C_4^3$};
      \node (c4d) at (-2+\xshift,4+\yshift) {$C_4^4$};
      \node (c4e) at (-1+\xshift,4+\yshift) {$C_4^5$};
 
      \node (d5) at (3+\xshift,4+\yshift) {$D_5$};   
      
      \node (c2a) at (-5+\xshift,2+\yshift) {$C_2^1$};
      \node (c2b) at (-4+\xshift,2+\yshift) {$C_2^2$};
      \node (c2c) at (-3+\xshift,2+\yshift) {$C_2^3$};
      \node (c2d) at (-2+\xshift,2+\yshift) {$C_2^4$};
      \node (c2e) at (-1+\xshift,2+\yshift) {$C_2^5$};

      \node (c5) at (3+\xshift,2+\yshift) {$C_5$};
      
      \node (c1) at (0+\xshift,0+\yshift) {$e$};

	 \draw[->,] (c1) to (c2a) ;
	 \draw[->,] (c1) to (c2b) ;
	 \draw[->,] (c1) to (c2c) ;
	 \draw[->,] (c1) to (c2d) ;
	 \draw[->,] (c1) to (c2e) ;
	 \draw[->,] (c1) to (c5) ;
  \draw[->,] (c2a) to (c4a) ;
	 \draw[->,] (c2b) to (c4b) ;
	 \draw[->,] (c2c) to (c4c) ;
	 \draw[->,] (c2d) to (c4d) ;
	 \draw[->,] (c2e) to (c4e) ;

	 \draw[->,] (c2a) to (d5) ;
	 \draw[->,] (c2b) to (d5) ;
	 \draw[->,] (c2c) to (d5) ;
	 \draw[->,] (c2d) to (d5) ;
	 \draw[->,] (c2e) to (d5) ;
  \draw[->,] (c5) to (d5) ;     

	 \draw[->,] (c4a) to (f5) ;
	 \draw[->,] (c4b) to (f5) ;
	 \draw[->,] (c4c) to (f5) ;
	 \draw[->,] (c4d) to (f5) ;
	 \draw[->,] (c4e) to (f5) ;
  \draw[->,] (d5) to (f5) ;

     \def \xshift {8} ;
     \def \yshift {0};
      \node (f5k) at (0+\xshift,6+\yshift) {$(F_5)$};
      \node (c4k) at (-2+\xshift,4+\yshift) {$(C_4)$};
      \node (d5k) at (2+\xshift,4+\yshift) {$(D_5)$};   
      \node (c2k) at (-2+\xshift,2+\yshift) {$(C_2)$};
      \node (c5k) at (2+\xshift,2+\yshift) {$(C_5)$};
      \node (c1k) at (0+\xshift,0+\yshift) {$(e)$};

	 \draw[->,] (c1k) to (c2k) ;
	 \draw[->,] (c1k) to (c5k) ;
	 \draw[->,] (c2k) to (c4k) ;

	 \draw[->,] (c2k) to (d5k) ;
      \draw[->,] (c5k) to (d5k) ;     

	 \draw[->,] (c4k) to (f5k) ;
      \draw[->,] (d5k) to (f5k) ;      

\end{tikzpicture}
\end{center}
As a warm-up, we first calculate the global dimension of \(IA(\Sub_{F_5}/F_5)\).
The poset \(\Sub_{F_5}/F_5\) is planar, so by Theorem \ref{thm:planaristwo}, the global dimension of its 
associated incidence algebra is bounded above by 2. 
We may calculate the \(\ext\)-group between the simple at \((D_5)\) and the simple at \((e)\) and see it is \(\bQ\) in degree 2.
Hence, the global dimension of \(IA(\Sub_{F_5}/F_5)\) is 2. 

Now we consider the global dimension of \(IA(\Sub_{F_5})\), the incidence algebra on \(\Sub_{F_5}\) 
where the conjugation action is not taken into account.
The poset \(\Sub_{F_5}\) is not planar. As before, by considering
\(D_5\) and its Frattini subgroup, \(e\), we see that the global dimension is at least 2. 
If the global dimension were three, it would have to arise from subgroups that are at least three arrows apart by 
Lemma \ref{lem:gdupperbound}. 
Hence we calculate the \(\ext\) group between the simple modules given by \(\bQ\) concentrated at \(F_5\) and 
the corresponding simple concentrated at its Frattini subgroup,  \(e\).  The corresponding interval is contractible, 
so we conclude the global dimension is 2.  
Thus the global dimension \(IA(\Sub_G)\) is realised by \(D_5\) and its Frattini subgroup, rather than  
\(F_5\) and its Frattini subgroup.
\end{example}

\begin{remark}
Informal calculations suggest that this global dimension of rational \(F_5\)-coefficient systems is 2, hence we do not expect the preceding example to give a suitable example for Conjecture \ref{conj:2}. We believe that making these calculations rigorous would require a new strategy, as Corollary \ref{cor:lower_bound_coeff} and the Abelian Sandwich Theorem, Theorem \ref{thm:sandwich}, do not hold if \(G\) is not abelian.
\end{remark}

\section{Examples}\label{sec:examples}

We end with two examples. The first 
illustrates that one needs to take a maximum over minimal elements in \(\Sub_{\langle H \rangle}^\cO\) in our formula for the global dimension in Proposition \ref{prop:gl_dim_cloud} (see also Definition \ref{def:n_h}).

\begin{example}
    Suppose \(G=C_{p^2q}\) for \(p,q\) distinct primes. 
    The subgroup lattice of \(G\) is as follows.
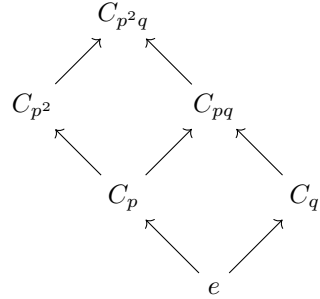
\begin{center}
\begin{tikzpicture}[scale=0.6,xscale=-1]
      \node (A) at (0,4) {\(C_{p^2q}\)};
	 \node (B) at (-2,2) {\(C_{pq}\)};
	 \node (C) at (2,2) {\(C_{p^2}\)};
      \node (D) at (0,0) {\(C_{p}\)};
      \node (E) at (-4,0) {\(C_{q}\)};
      \node (F) at (-2,-2) {\(e\)};
      \draw[<-,] (A) to (B) ;
	 \draw[<-,] (A) to (C) ;
	 \draw[<-,] (B) to (D) ;
	 \draw[<-,] (C) to (D) ;
     \draw[<-,] (B) to (E) ;
     \draw[<-,] (E) to (F) ;
     \draw[<-,] (D) to (F) ;
\end{tikzpicture}
\end{center}

Let \(\cO\) be a transfer system for \(G\) generated by one transfer \(e \rightarrow G\). In this case we get two inseparability classes \([e]\) and \([G]\). Notice that \(\Sub_{\langle e \rangle}^\cO\) consists only of the trivial subgroup \(e\) and \(\Sub_{\langle G \rangle}^\cO\) consists of all the other subgroups of \(G\), as in the picture below.

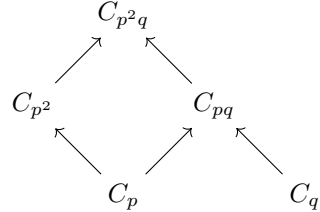
\begin{center}
\begin{tikzpicture}[scale=0.6,xscale=-1]
      \node (A) at (0,4) {\(C_{p^2q}\)};
	 \node (B) at (-2,2) {\(C_{pq}\)};
	 \node (C) at (2,2) {\(C_{p^2}\)};
      \node (D) at (0,0) {\(C_{p}\)};
      \node (E) at (-4,0) {\(C_{q}\)};
   %   \node (F) at (-2,-2) {\(p^2q C_{p^2q}\)};
      \draw[<-,] (A) to (B) ;
	 \draw[<-,] (A) to (C) ;
	 \draw[<-,] (B) to (D) ;
	 \draw[<-,] (C) to (D) ;
     \draw[<-,] (B) to (E) ;
 %    \draw[<-,] (E) to (F) ;
 %    \draw[<-,] (D) to (F) ;
\end{tikzpicture}
\end{center}

We can see that \(\Sub_{\langle G \rangle}^\cO\) has two minimal elements \(C_p\) and \(C_q\) which are incomparable. The number of the prime-power cyclic factors of \(G/C_q\) is one, and the number of the prime-power cyclic factors of \(G/C_p\) is two, thus the global dimension of \(\Coeff_{\langle G \rangle^{\cO},\bQ}^{\,G}\) is two. The global dimension of \(\Coeff_{\langle e \rangle^{\cO},\bQ}^{\,G}\) is zero, therefore by Theorem \ref{thm:main_thm_disklike} the global dimension of \(\OMackey_\bQ^G\) is the maximum of these two numbers, hence the global dimension in this case is two. 
\end{example}

The second example illustrates how the structure of inseparability classes changes as 
we add more generating transfers to a disk-like transfer system and how this affects the global dimension.  

\begin{example}
    Suppose \(G=C_{p^3q^2}\) with the subgroup lattice given by the picture below. 
    
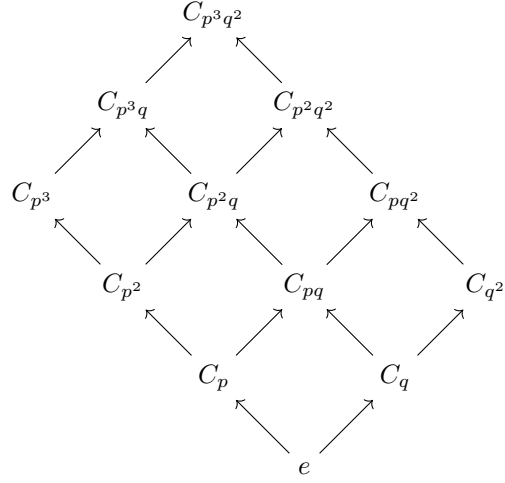
\begin{center}
\begin{tikzpicture}[scale=0.6]
      \node (A) at (0,4) {\(C_{p^3q^2}\)};
	 \node (B) at (-2,2) {\(C_{p^3q}\)};
	 \node (C) at (2,2) {\(C_{p^2q^2}\)};
      \node (D) at (0,0) {\(C_{p^2q}\)};
      \node (E) at (-4,0) {\(C_{p^3}\)};
      \node (F) at (-2,-2) {\(C_{p^2}\)};
       \node (G) at (0,-4) {\(C_{p}\)};
	 \node (H) at (2,-2) {\( C_{pq}\)};
	 \node (I) at (4,0) {\(C_{pq^2}\)};
      \node (J) at (2,-6) {\(e\)};
      \node (K) at (4,-4) {\( C_{q}\)};
      \node (L) at (6,-2) {\( C_{q^2}\)};
      
      \draw[<-,] (A) to (B) ;
	 \draw[<-,] (A) to (C) ;
	 \draw[<-,] (B) to (D) ;
	 \draw[<-,] (C) to (D) ;
     \draw[<-,] (B) to (E) ;
     \draw[<-,] (E) to (F) ;
     \draw[<-,] (D) to (F) ;
     \draw[<-,] (I) to (L) ;
	 \draw[<-,] (I) to (H) ;
	 \draw[<-,] (H) to (K) ;
	 \draw[<-,] (L) to (K) ;
     \draw[<-,] (H) to (G) ;
     \draw[<-,] (G) to (J) ;
     \draw[<-,] (K) to (J) ;
     \draw[<-,] (C) to (I) ;
     \draw[<-,] (D) to (H) ;
     \draw[<-,] (F) to (G) ;
     
\end{tikzpicture}
\end{center}    
    Take the disk-like transfer system \(\cO\) generated by 
    \(\{ C_{p^3} \rightarrow G\}\). This gives two inseparability classes of subgroups of \(G\). The corresponding \(\Sub_{\langle C_{p^3} \rangle}^\cO\) and \(\Sub_{\langle G \rangle}^\cO\) are illustrated on the picture below.

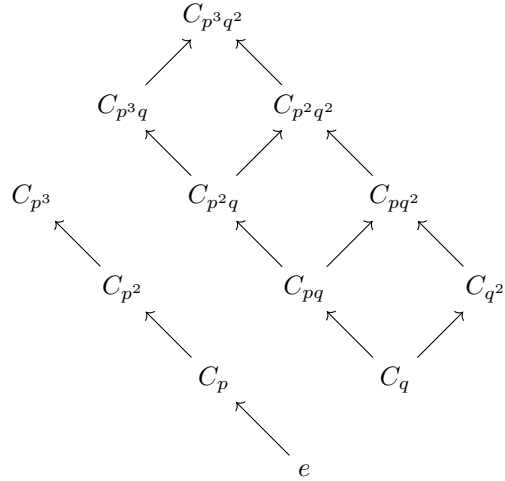
\begin{center}
\begin{tikzpicture}[scale=0.6]
      \node (A) at (0,4) {\(C_{p^3q^2}\)};
	 \node (B) at (-2,2) {\(C_{p^3q}\)};
	 \node (C) at (2,2) {\(C_{p^2q^2}\)};
      \node (D) at (0,0) {\(C_{p^2q}\)};
      \node (E) at (-4,0) {\(C_{p^3}\)};
      \node (F) at (-2,-2) {\(C_{p^2}\)};
       \node (G) at (0,-4) {\(C_{p}\)};
	 \node (H) at (2,-2) {\( C_{pq}\)};
	 \node (I) at (4,0) {\(C_{pq^2}\)};
      \node (J) at (2,-6) {\(e\)};
      \node (K) at (4,-4) {\( C_{q}\)};
      \node (L) at (6,-2) {\( C_{q^2}\)};
      \draw[<-,] (A) to (B) ;
	 \draw[<-,] (A) to (C) ;
	 \draw[<-,] (B) to (D) ;
	 \draw[<-,] (C) to (D) ;
  %   \draw[<-,] (B) to (E) ;
     \draw[<-,] (E) to (F) ;
 %    \draw[<-,] (D) to (F) ;
     \draw[<-,] (I) to (L) ;
	 \draw[<-,] (I) to (H) ;
	 \draw[<-,] (H) to (K) ;
	 \draw[<-,] (L) to (K) ;
  %   \draw[<-,] (H) to (G) ;
     \draw[<-,] (G) to (J) ;
 %    \draw[<-,] (K) to (J) ;
     \draw[<-,] (C) to (I) ;
     \draw[<-,] (D) to (H) ;
     \draw[<-,] (F) to (G) ;
\end{tikzpicture}
\end{center}
The global dimension of \(\OMackey_\bQ^G\) is two by  Theorem \ref{thm:main_thm_disklike} since the number of prime-power cyclic factors of \(C_{p^3}/e\) is one and for \(C_{p^3q^2}/C_q\) it is two.

Take the disk-like transfer system \(\cO'\) generated by \(\{C_{p^3} \rightarrow G, \ C_{pq} \rightarrow G\}\). Notice that \(\cO\leq \cO'\) in the poset of transfer systems (ordered by inclusion). By the restriction axiom of transfer systems, there is also a third transfer in \(\cO'\) ending at \(G\), namely  \(C_{p} \rightarrow G\). These transfers give four inseparability classes of subgroups of \(G\). The picture shows \(\Sub_{\langle H \rangle}^\cO\) for \(H=C_p,C_{pq},C_{p^3}, C_{p^3q^2}\). 
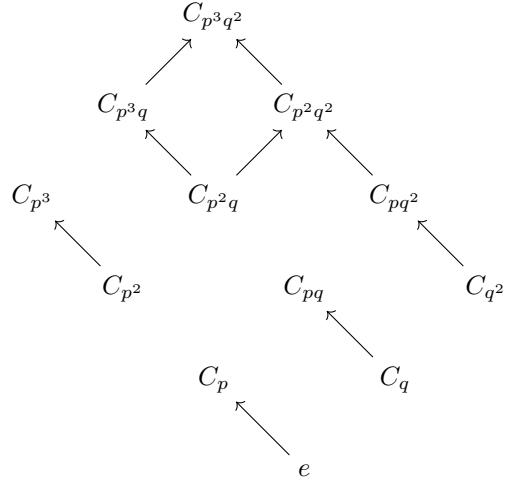
\begin{center}
\begin{tikzpicture}[scale=0.6]
      \node (A) at (0,4) {\(C_{p^3q^2}\)};
	 \node (B) at (-2,2) {\(C_{p^3q}\)};
	 \node (C) at (2,2) {\(C_{p^2q^2}\)};
      \node (D) at (0,0) {\(C_{p^2q}\)};
      \node (E) at (-4,0) {\(C_{p^3}\)};
      \node (F) at (-2,-2) {\(C_{p^2}\)};
       \node (G) at (0,-4) {\(C_{p}\)};
	 \node (H) at (2,-2) {\( C_{pq}\)};
	 \node (I) at (4,0) {\(C_{pq^2}\)};
      \node (J) at (2,-6) {\(e\)};
      \node (K) at (4,-4) {\( C_{q}\)};
      \node (L) at (6,-2) {\( C_{q^2}\)};
      \draw[<-,] (A) to (B) ;
	 \draw[<-,] (A) to (C) ;
	 \draw[<-,] (B) to (D) ;
	 \draw[<-,] (C) to (D) ;
  %   \draw[<-,] (B) to (E) ;
     \draw[<-,] (E) to (F) ;
 %    \draw[<-,] (D) to (F) ;
     \draw[<-,] (I) to (L) ;
%	 \draw[<-,] (I) to (H) ;
	 \draw[<-,] (H) to (K) ;
%	 \draw[<-,] (L) to (K) ;
  %   \draw[<-,] (H) to (G) ;
     \draw[<-] (G) to (J) ;
 %    \draw[<-,] (K) to (J) ;
     \draw[<-,] (C) to (I) ;
 %    \draw[<-,] (D) to (H) ;
%     \draw[<-,] (F) to (G) ; 
\end{tikzpicture}
\end{center}
   We claim the global dimension of  \(\OpMackey_\bQ^G\) remains  two. Consider the largest inseparability class, which has
   \(C_{p^3q^2}\) as the maximal element and \(C_{p^2q}\) and \(C_{q^2}\) as minimal elements. 
   The number of prime-power cyclic factors of \(C_{p^3q^2}/C_{p^2q}\) and \(C_{p^3q^2}/C_{q^2}\)
   are two and one respectively, proving the claim. 
\end{example}

To reduce the global dimension, we add the transfer \(C_{p^2q} \rightarrow G\). That is, consider the disk-like transfer system \(\cO''\) generated by \(\{C_{p^3} \rightarrow G, \ C_{pq} \rightarrow G, \ C_{p^2q} \rightarrow G\}\). This transfer system has more inseparability classes than for \(\cO'\) and in particular, \(G\) and \(C_{p^2q}\) will be in different inseparability classes. This means there is no possibility for an inseparability class containing \(L \geqslant K\) with \(L/K\) having two prime-power cyclic factors, thus the global dimension is less than two. This transfer system cannot be the complete transfer system, as the minimal generating set for the complete transfer system of \(C_{p^3 q^2}\) is five by \cite[Theorem B]{ABBSWC}. By Corollary \ref{cor:dimzerocase}, we conclude that the  global dimension is not 0, thus it is 1.  We note that we did not need to explicitly calculate the  inseparability classes or the whole of the transfer system to deduce this.

\bibliographystyle{alpha}
\bibliography{biblio}
\end{document}